%% file: ConvergenceBurgersParticle.tex
\documentclass[a4paper, twoside, openright, extrafontsizes]{article}

\input{preambuleenglish}
\usepackage{appendix}

\makeatletter
\newcommand*{\inlineequation}[2][]{%
  \begingroup
    \refstepcounter{equation}%
    \ifx\\#1\\%
    \else
      \label{#1}%
    \fi
    \relpenalty=10000 %
    \binoppenalty=10000 %
    \ensuremath{%
      #2%
    }%
    ~\@eqnnum
  \endgroup
}
\makeatother

\title{Convergence of finite volumes schemes for the coupling between the inviscid Burgers equation and a particle}
\author{Nina Aguillon\thanks{Universit\'e Paris Sud}, Frédéric Lagoutière\thanks{Université Paris Sud}, Nicolas Seguin\thanks{Université Pierre et Marie Curie}}

\begin{document}
\maketitle
\tableofcontents

\abstract{In this paper, we prove the convergence of a class of finite volume schemes for the model of coupling between a Burgers fluid and a pointwise particle introduced in~\cite{LST08}. In this model, the particle is seen as a moving interface through which an interface condition is imposed, which links the velocity of the fluid on the left and on the right of the particle and the velocity of the particle (the three quantities are all not equal in general). The total impulsion of the system is conserved through time.

The proposed schemes are consistent with a ``large enough''  part of the interface conditions. The proof of convergence is an extension of the one of~\cite{AS12} to the case where the particle moves under the influence of the fluid. It yields two main difficulties: first, we have to deal with time-dependent flux and interface condition, and second with the coupling between and ODE and a PDE. }
 
\noindent  \textbf{Key phrases:} Fluid-particle interaction; Burgers equation; Non-conservative coupling; moving interface; convergence of finite volume schemes; PDE-ODE coupling

  
\section{Introduction}

We study the numerical convergence of finite volume schemes for the Cauchy problem
\begin{equation} \label{eq:CauchyPb} 
\begin{cases}
 \partial_{t} u + \partial_{x} \frac{u^{2}}{2} = - \lambda (u-h'(t)) \delta_{h(t)}(x), \\
 m_{p} h''(t)= \lambda (u(t,h(t))-h'(t)), \\
 u_{| t=0}= u^{0}, h(0)= h^{0}, h'(0)=v^{0}.
\end{cases}
\end{equation}
It models the behavior of a pointwise particle of position $h$, velocity $h'$ and acceleration $h''$ with mass $m_{p}$, immersed into a ``fluid,'' whose velocity at time $t$ and point $x$ is $u(t,x)$. The velocity of the fluid is assumed to follow the inviscid Burgers equation. This system is fully coupled: the fluid exerts a drag force $D=  \lambda (u(t,h(t))-h'(t))$ on the particle, where $\lambda$ is a positive friction parameter. By the action--reaction principle, the particle exerts the force $-D$ on the particle. The interaction is local: it applies only at the point where the particle is. This friction force tends to bring the velocities of the fluid and the particle closer to each other: as $\lambda$ is positive, the particle accelerates if $u(t,h(t))$ is larger than $h'(t)$ and vice-versa. This toy model was introduced in~\cite{LST08} (see also~\cite{BCG13} and~\cite{A12} for related problems).  In contrast with the model studied in~\cite{VZ03},~\cite{H05} and~\cite{VZ06}, the particle and the fluid do not share the same velocity and the fluid is inviscid.
 In particular the fluid velocity is typically discontinuous through the particle. It yields to issues to define correctly the product $(u-h') \delta_{h}$ and the ODE for the particle in system~\eqref{eq:CauchyPb}. To do so, the idea is to regularize the Dirac measure in~\eqref{eq:CauchyPb}, and to remark that the values of the fluid velocity on both sides of this thickened particle are independent of the regularization. It allows to reformulate System~\eqref{eq:CauchyPb} as an interface problem, where the traces around the particle $u_{-}(t)= \lim_{x \rightarrow h(t)^{-}} u(t, x)$ and $u_{+}(t)= \lim_{x \rightarrow h(t)^{+}} u(t, x)$ must belong to a set $\mathcal{G}_{\lambda}(h'(t))$, which takes into account the interface conditions. This study was done in details in~\cite{LST08}. The germ is defined as follow.
 \begin{defi} \label{def:germ1}
For any given speed $v\in\R$, the germ at speed $v$, $\mathcal{G}_{\lambda}(v)$, is the set of all $(u_{-}, u_{+})$ in $\R^{2}$ such that
 $$(u_{-}, u_{+}) \in \mathcal{G}_{\lambda}^{1} \cup  \mathcal{G}_{\lambda}^{2}(v) \cup  \mathcal{G}_{\lambda}^{3}(v),$$
 where
 $$ \mathcal{G}_{\lambda}^{1} =\left\{ (u_{-}, u_{+}) \in \R^{2} : u_{-}= u_{+}+ \lambda \right\}, $$
 $$ \mathcal{G}_{\lambda}^{2}(v) =\left\{ (u_{-}, u_{+}) \in \R^{2} : v \leq u_{-} \leq v+\lambda, \,  v- \lambda \leq u_{+} \leq v \ \text{ and } \ u_{-}-u_{+}<\lambda \right\}, $$
 and
 $$ \mathcal{G}_{\lambda}^{3}(v) =\left\{ (u_{-}, u_{+}) \in \R^{2} :  -\lambda \leq u_{+} + u_{-} -2v \leq \lambda \ \text{ and } \ u_{-}-u_{+}>\lambda \right\}. $$
\end{defi}
The germ $\mathcal{G}_{\lambda}(0)$ and its partition are depicted on Figure~\ref{F:zones2} on the left (note that the germ $\mathcal{G}_\lambda(v)$ is the translation of $\mathcal{G}_\lambda(0)$ by the vector $(v,v)$). Here, we choose a slightly different partition of the germ than in~\cite{AS12} and~\cite{ALST13}, which is depicted on the right of Figure~\eqref{F:zones2}. The reason is that we are able to find a class of finite volume schemes which are consistent with $\mathcal{G}_{\lambda}^{1} \cup \mathcal{G}_{\lambda}^{2}(0)$ with this choice, but not with the original partition. However, the essential property that $\mathcal{G}_{\lambda}^{1} \cup \mathcal{G}_{\lambda}^{2}(0)$ is a \emph{maximal} part of the germ still holds true with the partition of Definition~\ref{def:germ1} (more details are given in Definition~\ref{def:MaxSub} and Proposition~\ref{prop:MaxSub}).
\begin{psfrags}
 \psfrag{u_{-}}{$u_{-}$}
 \psfrag{u_{+}}{$u_{+}$}
 \psfrag{v}{$v$}
 \psfrag{-l}{$-\lambda$}
 \psfrag{l}{$\lambda$}
 \psfrag{G1}{$\mathcal{G}_{\lambda}^{1}$}
 \psfrag{G2}{$\mathcal{G}_{\lambda}^{2}(0)$}
 \psfrag{G3}{$\mathcal{G}_{\lambda}^{3}(0)$}
\begin{figure}[htp]
\centering
 \includegraphics[width=\linewidth]{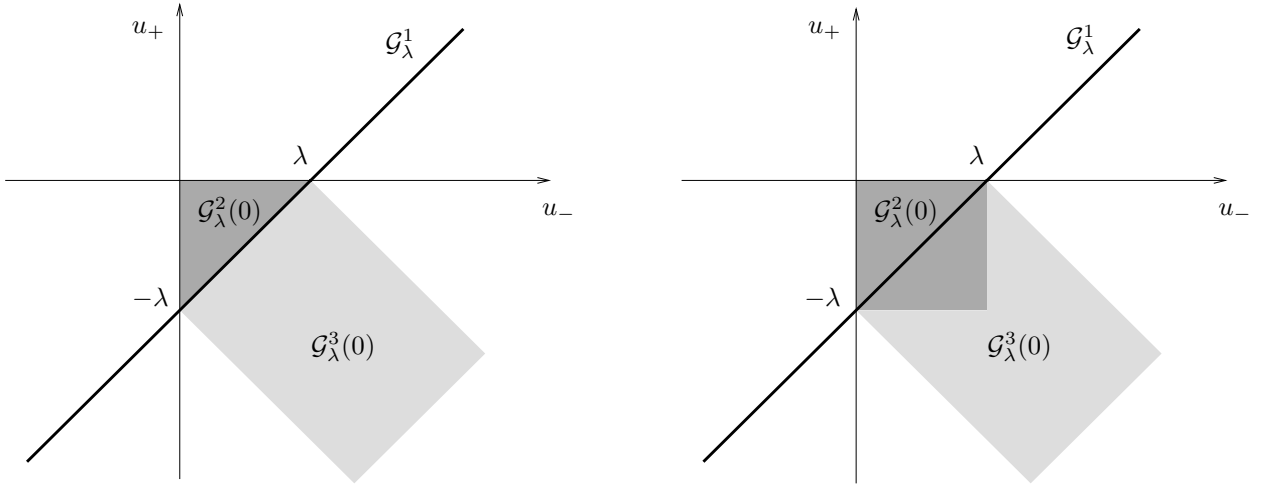}
 \caption[Two different partitions of the germ.]{The germ for a motionless particle and its partitions. Left: the partition used in this work. Right: the partition used in~\cite{AS12} and~\cite{ALST13}.} \label{F:zones2}
\end{figure}
\end{psfrags}
Once the germ has been defined, System~\eqref{eq:CauchyPb} is defined as an interface problem. The equation on the particle is reformulated to keep the conservation of total momentum
$$ m_{p} h'(t) + \int_{\R} u(t,x) dx $$
which holds formally in~\eqref{eq:CauchyPb}. In~\cite{LST08}, an entropy inequality that takes into account the particle is also derived.
\begin{defi} \label{def:Traces}
 A pair $(u,h)$ of functions in $L^{\infty}(\R_{+} \times \R) \times W^{2, \infty}(\R_{+})$ is a solution of~\eqref{eq:CauchyPb} with initial data $u^{0}$ in $L^{\infty}(\R)$ and ($h^{0}, v^{0}) \in \R^{2}$ if:
 \begin{itemize}
 \item the function $u$ is an entropy weak solution of the Burgers equations on the sets $\{(t,x), \, x<h(t)\}$ and $\{(t,x), \, x>h(t)\}$, 
 \item for almost every positive time $t$,
 \begin{equation} \label{eq:ODE}
 m_{p} h''(t)= (u_{-}(t)- u_{+}(t)) \left( \frac{u_{-}(t)+u_{+}(t)}{2} -h'(t) \right)
\end{equation}
 and
 $$ (u_{-}(t), u_{+}(t)) \in \mathcal{G}_{\lambda}(h'(t)). $$
\end{itemize}
\end{defi}
This definition requires the existence of traces along the particle's trajectory $h$. It follows from the works of Panov~\cite{P07} and Vasseur~\cite{V01}.
When the particle is motionless, well-posedness in the $BV$ setting was proved in~\cite{AS12}, while for the fully coupled system~\eqref{eq:CauchyPb}, it is proved in~\cite{ALST10} and~\cite{ALST13}.

Remark that Definition~\ref{def:Traces} is not suitable to prove convergence of finite volume schemes in a general framework. Indeed, a scheme can create a numerical boundary layer near the particle, of several cells width. It does not prevent the scheme from converging in, say, $L^{\infty}_{loc}$ in time and $L^{1}$ in space; but in that case we cannot expect the numerical traces to converge to their correct values. Nevertheless we will prove the convergence of some schemes that create such boundary layers. The key point is to use, instead of Definition~\ref{def:germ1}, an equivalent definition which does not contain the traces of $u$. We begin with some properties useful to decide if a pair $(c_{-}, c_{+})$ belongs to the germ $\mathcal{G}_{\lambda}(v)$. We adopted the vocabulary of the theory of conservation law with discontinuous flux function of~\cite{AKR10} and~\cite{AKR11}.

 In the sequel, we denote by  $\Phi_{v}$ the so-called Kruzhkov entropy flux associated with $f_{v}(u)= \frac{u^{2}}{2}-v u$:
$$ \Phi_{v}: \, 
\begin{array}{ccc}
 \R^{2} & \longrightarrow & \R \\
 (a,b) & \longmapsto & \sign(a-b) \left( \left( \frac{a^{2}}{2}-va \right)-\left( \frac{b^{2}}{2}-vb\right)\right)
\end{array}
$$
and we define
$$ \Xi_{v}: \, 
\begin{array}{ccc}
 \R^{2} \times \R^{2} & \longrightarrow & \R \\
( (a_{-},a_{+}), (b_{-}, b_{+}) ) & \longmapsto & \Phi_{v}(a_{-},b_{-}) -  \Phi_{v}(a_{+},b_{+})
\end{array}
$$
\begin{prop} \label{def:GermPos}
 If both $(a_{-}, a_{+})$ and $(b_{-}, b_{+})$ belong to $\mathcal{G}_{\lambda}(v)$, then
 $$\Xi_{v}( (a_{-},a_{+}), (b_{-}, b_{+}) ) \geq 0.  $$
 Conversely, if $(a_{-}, a_{+})$ is such that
  $$\forall(b_{-}, b_{+}) \in \mathcal{G}_{\lambda}(v), \quad  \Xi_{v}( (a_{-},a_{+}), (b_{-}, b_{+}) ) \geq 0,  $$
  then $(a_{-}, a_{+})$ belongs to the germ.
\end{prop}
\begin{defi} \label{def:MaxSub}
 A subset $\mathcal{H}_{\lambda}(v)$ of $\mathcal{G}_{\lambda}(v)$ is said to be maximal if any $(a_{-},a_{+})$ that satisfies
 \begin{equation} \label{eq:GermAppCrit}
 \forall(b_{-}, b_{+}) \in \mathcal{H}_{\lambda}(v), \quad \Xi_{v}( (a_{-},a_{+}), (b_{-}, b_{+}) ) \geq 0
\end{equation}
belongs to the germ $\mathcal{G}_{\lambda}(v)$.
\end{defi}
We will prove in Proposition~\ref{prop:MaxSub} that  $\mathcal{G}_{\lambda}^{1} \cup \mathcal{G}_{\lambda}^{2}(v)$ is maximal. In the sequel $\mathcal{H}_{\lambda}(v)$ always denotes a maximal part of $\mathcal{G}_{\lambda}(v)$. We now focus on alternative traceless characterizations of entropy solutions.
For all $(c_{-},c_{+})$ we denote by $c$ the piecewise constant function
$$ c(t,x)= c_{-} \mathbf{1}_{x < h(t)} +  c_{+} \mathbf{1}_{x \geq h(t)}, $$
and by $ \dist_{1}(a, X)$ the $L^{1}$-distance of a point $a:=(a_{-},a_{+})$ of $\R^{2}$ to a set $X$ included in $\R^{2}$:
$$ \dist_{1}((a_{-},a_{+}), X) = \inf_{(x_{-},x_{+}) \in X} |a_{-}-x_{-}|+|a_{+}-x_{+}|. $$
\begin{prop}
Let $h$ be a function of $W^{2,\infty}_{loc}(\R_{+})$ and let $u$ be a function of $L^{\infty}_{loc}(\R_{+} \times \R)$, which is an entropy solution of the Burgers equation on the sets $\{(t,x), \, x<h(t)\}$ and $\{(t,x), \, x>h(t)\}$. The following assertions are equivalent.
\begin{itemize}
 \item For almost every time $t>0$, $(u_{-}(t), u_{+}(t))$ belongs to $\mathcal{G}_{\lambda}(h'(t))$.
 \item For almost every time $t>0$, for all $(c_{-},c_{+})\in \R^{2}$, there exist $\delta \in (0,t)$ and a constant $A$ depending only on $||u^{0}||_{\infty}$, $\lambda$, $(c_{-},c_{+})$ and  $||h'||_{\infty}$ such that for every nonnegative function $\varphi$ in $\mathcal{C}_{0}^{\infty}((t-\delta, t+ \delta) \times \R)$,
\begin{equation} \label{eq:DefFluid1}
\begin{aligned}
 \dis \int_{\R_{+}} \int_{\R} & |u-c|(s,x) \partial_{t} \varphi (s,x-h(s))+  \Phi_{h'(t)}(u,c)(s,x) \partial_{x} \varphi(s,x-h(s)) dx\, ds \\
 	&  \geq -A \int_{\R_{+}} \dist_{1}((c_{-},c_{+}), \mathcal{H}_{\lambda}(h'(s)) ) \varphi(s,0) \, ds.
\end{aligned}
\end{equation}
\end{itemize}
\end{prop}
\begin{proof} For the sake of completeness we reproduce here the main ingredients of the proof that can be found in~\cite{AS12}. Let $\varphi$ be in $\mathcal{C}_{0}^{\infty}((t-\delta, t+ \delta) \times \R)$, where $\delta$ belongs to $(0,t)$. For positive $\eps$, we introduce the  function 
 $$ \zeta_{\eps}(z)= 1- \min(1, |z|/ \eps), $$
whose support is $(-\eps, \eps)$. The support of the function 
 $$ \psi_{\eps}(t,x)= (1-\zeta_{\eps}) \varphi(t,x-h(t)) $$
 is included in $\{(t,x), \, t>0, \, x \neq h(t)\}$. The function $u$ is a entropy solution of the Burgers equation on the sets $\{(t,x), \, x<h(t)\}$ and $\{(t,x), \, x>h(t)\}$, thus for all real $\kappa$,
$$ 
 \iint_{\R_{+}\times \R} |u(s,x)-\kappa| \partial_{s} \psi_{\eps}(s,x) + \Phi_{0}(u(s,x), \kappa) \partial_{x} \psi_{\eps} dx \, ds \geq 0.
 $$
But $\partial_{s} \psi_{\eps}(s,x)= \partial_{s} ((1-\zeta_{\eps}) \varphi) (s, x-h(s)) -h'(s) \partial_{x}  ((1-\zeta_{\eps}) \varphi)(s,x-h(s))$, and we using the fact that
$$ \Phi_{v}(a,b)= \Phi_{0}(a,b)-v |a-b|,$$
we obtain
$$ 
 \iint_{\R_{+}\times \R} |u-c|(s,x) (\partial_{s} (1-\zeta_{\eps}) \varphi) (s,x-h(s)) + \Phi_{0}(u, c)(s,x) \partial_{x}((1-\zeta_{\eps}) \varphi)(s, x-h(s)) dx \, ds \geq 0.
 $$
 Thus we have
 $$ 
\begin{aligned}
  & \iint_{\R_{+} \times \R}  |u-c|(s,x) \partial_{t} \varphi (s,x-h(s))+  \Phi_{h'(t)}(u,c)(s,x) \partial_{x} \varphi(s,x-h(s)) dx\, ds \\
  & \geq \liminf_{\eps \rightarrow 0} \iint_{\R_{+} \times \R} |u-c|(s,x) (\partial_{t} ( \zeta_{\eps} \varphi)) (s,x-h(s))+  \Phi_{h'(t)}(u,c)(s,x) (\partial_{x}  ( \zeta_{\eps} \varphi))(s,x-h(s)) dx \, ds \\
&=  \int_{\R_{+}} \Phi_{h'(s)}(u_{-}(s), c_{-})- \Phi_{h'(s)}(u_{+}(s), c_{+}) \varphi (s,0) ds \\
&= \int_{\R_{+}} \Xi_{h'(s)}((u_{-}(s), u_{+}(s)), (c_{-},c_{+})) \varphi (s,0) ds 
\end{aligned}
$$ 
%
For all $s$ for which the pair $(u_{-}(s), u_{+}(s))$ exists and belongs to $\mathcal{G}_{\lambda}(h'(s))$, we denote by $(\tilde{c}_{-}(s), \tilde{c}_{+}(s))$ a $L^{1}$-projection of $(c_{-},c_{+})$ on $\mathcal{H}_{\lambda}(h'(s))$. We have
$$ 
\begin{aligned}
 \Xi_{h'(s)}& ((u_{-}(s), u_{+}(s)), (c_{-}(s), c_{+}(s)))
  \geq  \Xi_{h'(s)}((u_{-}(s), u_{+}(s)), (\tilde{c}_{-}(s), \tilde{c}_{+}(s))) \\
 & - |\Xi_{h'(s)}((u_{-}(s), u_{+}(s)), (c_{-}(s), c_{+}(s)))- \Xi_{h'(s)}((u_{-}(s), u_{+}(s)), (\tilde{c}_{-}(s), \tilde{c}_{+}(s)))|.
\end{aligned}
$$
Since $(\tilde{c}_{-}(s), \tilde{c}_{+}(s))$ belongs to $\mathcal{H}_{\lambda}(h'(s))$, Proposition~\ref{def:GermPos} yields
$$ \int_{\R_{+}} \Xi_{h'(s)}((u_{-}(s), u_{+}(s)), (\tilde{c}_{-}(s),\tilde{c}_{+}(s))) \varphi (s,0) ds \geq 0. $$
On the other hand
$$ 
\begin{aligned}
 |\Xi_{h'(s)}&((u_{-}(s), u_{+}(s)), (c_{-}(s), c_{+}(s)))- \Xi_{h'(s)}((u_{-}(s), u_{+}(s)), (\tilde{c}_{-}(s), \tilde{c}_{+}(s)))| \\
 & \leq |\Phi_{h'(s)}(u_{-}(s), c_{-}(s)) - \Phi_{h'(s)}(u_{-}(s), \tilde{c}_{-}(s))| +  |\Phi_{h'(s)}(u_{+}(s), c_{+}(s)) - \Phi_{h'(s)}(u_{+}(s), \tilde{c}_{+}(s))| 
\end{aligned}
$$
which is smaller than a constant depending only on $|| h' ||_{\infty}$, $||u||_{\infty}$, $c$ and $\lambda$ (since $c \mapsto \tilde{c}$ depends on $\lambda$), multiplied by the $L^{1}$-distance between $(c_{-},c_{+})$ and $(\tilde{c}_{-}(s), \tilde{c}_{+}(s))$, and we obtain the result. 

Conversely, using a sequence of test functions $\varphi$ concentrating at a time $t$ for which $u$ has traces in Proposition~\ref{eq:DefFluid1}, we obtain that for all $(c_{-},c_{+})$ in $\mathcal{H}_{\lambda}(h'(t))$,
$$\Xi_{h'(t)}((u_{-}(t), u_{+}(t)), (c_{-},c_{+})) \geq 0, $$
and thus by Proposition~\ref{def:GermPos}, $(u_{-}(t), u_{+}(t))$ belongs to the germ $\mathcal{G}_{\lambda}(h'(t))$.
\end{proof}

\begin{prop}
Let $u$ in $L^{\infty}_{loc}(\R_{+} \times \R)$ be a solution of the Burgers equation on the sets $\{(t,x), \, x<h(t)\}$ and $\{(t,x), \, x>h(t)\}$. Consider a function $h$ in $W^{2,\infty}_{loc}(\R_{+})$ which verifies ~\eqref{eq:ODE}  with initial data $h(0)=h^{0}$ and $h'(0)=v^{0}$ almost everywhere if and only if for all $\xi \in \mathcal{C}_{0}^{\infty}([0,T))$ and for all $\psi \in  \mathcal{C}_{0}^{\infty}(\R)$ such that $\psi(0)=1$
 \begin{equation} \label{eq:DefPart}
 \begin{aligned}
 -\int_{0}^{T} m_{p} h'(t) \xi'(t) dt = \ & m_{p} v^{0} \xi(0) +\int_{\R} \int_{0}^{T} \frac{u^{2}}{2}(s,x) \xi(s) \psi'(x-h(s)) ds \, dx \\
 	&+ \int_{\R} \int_{0}^{T} u(s,x) [ \xi'(s)-h'(s) \psi'(x-h(s))]ds \, dx \\
	&+ \int_{\R} u^{0}(x) \psi(x-h(0)) \xi(0) dx.
\end{aligned}
\end{equation}
\end{prop}
\begin{proof} This characterization were proved in~\cite{ALST10}. It follows from the application of the Green--Gauss theorem and the fact that $u$ is an entropy solution of the Burgers equation away from the particle:
$$ 
\begin{aligned}
\int_{0}^{T} \int_{\R}  & \frac{u^{2}}{2}(s,x) \xi(s) \psi'(x-h(s)) + u(s,x) [ \xi'(s)-h'(s) \psi'(x-h(s))] ds \, dx \\
 & = \int_{0}^{T} \int_{\R}  \frac{u^{2}}{2}(s,x) \partial_{x} ( \xi \psi (x-h(s)))  + u(s,x) \partial_{s} ( \xi \psi (x-h(s))) ds \, dx \\
 &=  -\int_{0}^{T} \int_{x \neq h(t)} \left(\partial_{x} \frac{u^{2}}{2} + \partial_{t} u \right) (\xi \psi) dx \, ds - \int_{\R} u^{0}(x) \psi(x-h(0)) \xi(0)dx \\
 & \hphantom{=  \,}+ \int_{0}^{T} \xi(s) \left( \left( \frac{u_{-}^{2}(s)}{2} - h'(s) u_{-}(s)\right) -\left( \frac{u_{+}^{2}(s)}{2} - h'(s) u_{+}(s)\right) \right) ds \\
&= \int_{0}^{T} m_{p} h''(s) \xi(s) ds - \int_{\R} u^{0}(x) \psi(x-h(0)) \xi(0) dx.
\end{aligned}
$$
\end{proof}

We now present the family of finite volume schemes for which we prove convergence. The proof follows the guidelines of the Lax--Wendroff theorem. In Section~\ref{S:Extract}, we obtain a $BV$ bound on the fluid velocity and a $W^{2, \infty}$ bound on the particle's trajectory that allows to extract convergent subsequences in $L^{1}_{loc}(\R_{+} \times \R)$ and $W^{1, \infty}_{loc}(\R_{+})$. The difficulties are to treat numerically the interface conditions enclosed in the germ and the coupling between an ODE and a PDE. 
More precisely:
\begin{itemize}
 \item First, we have to take into account at the numerical level the interface condition of Definition~\ref{def:germ1}. We will use schemes that preserves a ``sufficiently large'' part of the germ.
 \item Second, to deal with a moving particle. It is crucial that the particle lies at an interface of the mesh at the beginning of the time step. To do so and avoid the problem of the replacement of the particle, we use a mesh that tracks the particle and we update the particle's velocity by conservation of total impulsion.
 \end{itemize}
 Let us fix a time step $\Delta t$ and a space step $\Delta x$. In the sequel we suppose that the time step and the space step are proportional, and we denote by $\mu= \frac{\Delta t}{\Delta x}$ their ratio. We propose to approximate the solution of~\eqref{eq:CauchyPb} with a finite volume scheme. We use a mesh that follows the particle, which is placed between the cells numbered $0$ and $1$. The speed of the particle is approximated by a piecewise constant $(v^{n})_{n \in \N}$. Given the solution a time $n \Delta t$: we consider that the particle has constant velocity $v^{n}$ on the whole time step $(n \Delta t, (n+1) \Delta t)$ to update the fluid velocity, then we update $v^{n}$ by conservation of the total impulsion. The interface $1/2$ where the particle lies is special, and we have to use appropriate fluxes at this interface. Due to the source term, the equation is not conservative around the particle, thus we have two different fluxes $f_{1/2}^{n,-}$  and $f_{1/2}^{n,+}$ on the left and on the right of the particle respectively. Away from the particle, Equation~\eqref{eq:CauchyPb} writes as a scalar conservation law, and we can use any standard flux for the Burgers equation. 
The scheme is initialized with
$$ \forall j \in \Z, \ u_{j}^{0}= \frac{1}{\Delta x} \int_{x_{j-1/2}^{0}}^{x_{j+1/2}^{0}}u^{0}(x) \, dx. $$
From the integration of the first equation of~\eqref{eq:CauchyPb} on the space time cell
 $$ \mathcal{C}_{j}^{n}= \{(n \Delta t+s,x_{j-1/2}^{n}+y+s v^{n}), 0 \leq s < \Delta t, 0 \leq y < \Delta x \}, $$
 we obtain the finite volume scheme
 \begin{equation} \label{eq:CoupledScheme}
 \begin{cases}
 u_{j}^{n+1} &= u_{j}^{n}- \mu (f_{j+1/2}^{n}(v^{n}) - f_{j-1/2}^{n} (v^{n}))  \text{ for $j \in \Z, j \notin \{0, 1\}$}, \\
 u_{0}^{n+1} &= u_{0}^{n}- \mu (f_{1/2,-}^{n}(v^{n}) - f_{-1/2}^{n}(v^{n})), \\
 u_{1}^{n+1} &= u_{1}^{n}- \mu (f_{3/2}^{n}(v^{n}) - f_{1/2,+}^{n}(v^{n})), \\
 v^{n+1} & = v^{n} + \frac{\Delta t}{m_{p}} (f_{1/2,-}^{n}(v^{n})-f_{1/2,+}^{n}(v^{n})), \\
 x_{j}^{n+1}& =x_{j}^{n}+ v^{n}\Delta t.
\end{cases}
\end{equation}
Here we emphasized the dependency of the flux on the particle's velocity.  In the sequel we denote by $u_{\Delta t}$ the constant by cell function
\begin{equation} \label{def:udelta}
  u_{ \Delta t} (t,x) = u_{j}^{n} \ \  \text{ if } \ \  (t,x) \in \mathcal{C}_{j}^{n}.
\end{equation}
and  by $v_{\Delta t}$ and $h_{\Delta t}$ the constant and linear by cell functions:
\begin{equation} \label{def:hdelta}
 \begin{cases}
 v_{ \Delta t} (t) &= v^{n} \ \  \text{ if } n \Delta t \leq t < (n+1) \Delta t,  \\
 h_{ \Delta t} (t) &= h^{0} +   \Delta t \sum_{m=0}^{n-1} v^{m} + v^{n}(t-n \Delta t)  \ \  \text{ if } n \Delta t \leq t < (n+1) \Delta t. 
\end{cases}
\end{equation}

Another way to proceed is to performed the change of variable 
$$ \tilde{u}(t,x)= u(t, x+h(t)) $$
in~\eqref{eq:CauchyPb}. This function verifies the PDE
\begin{equation} \label{eq:redressement}
 \partial_{t} \tilde{u} + \partial_{x} \left(  \frac{\tilde{u}^{2}}{2}- h'(t) \tilde{u} \right) = -\lambda (\tilde{u}-h') \delta_{0}(x)
\end{equation}
The particle is now motionless but the flux depends on time. We denote by $f_{v}(u)= \frac{u^{2}}{2}-vu$. Integrating~\eqref{eq:redressement} on $[n \Delta t, (n+1) \Delta t] \times [x_{j-1/2}^{0}, x_{j+1/2}^{0}]$, and using special flux around the particle (still placed at interface $1/2$), we obtain the finite volume scheme
 \begin{equation} \label{eq:CoupledSchemeR}
 \begin{cases}
 \tilde{u}_{j}^{n+1} &= \tilde{u}_{j}^{n}- \mu (f_{j+1/2}^{v^{n},n} - f_{j-1/2}^{v^{n},n})  \text{ for } j \in \Z \setminus \{0, 1\}, \\
 \tilde{u}_{0}^{n+1} &= \tilde{u}_{0}^{n}- \mu (f_{1/2}^{v^{n}, n -} - f_{-1/2}^{v^{n},n}), \\
 \tilde{u}_{1}^{n+1} &= \tilde{u}_{1}^{n}- \mu (f_{3/2}^{v^{n}, n} - f_{1/2}^{v^{n}, n +}), \\
 v^{n+1} & = v^{n} + \frac{\Delta t}{m_{p}} (f_{1/2}^{v^{n},n-}-f_{1/2}^{v^{n},n+}). 
\end{cases}
\end{equation}
The two points of view are illustrated on Figure~\ref{F:POV}. 
\begin{psfrags}
 \psfrag{eq1}{$\partial_{t} u + \partial_{x} \frac{u^{2}}{2}= -\lambda (u-h'(t))\delta_{h(t)}(x)$}
 \psfrag{eq2}{$\partial_{t} \tilde{u} + \partial_{x} \left(  \frac{\tilde{u}^{2}}{2}- h'(t) \tilde{u} \right) = -\lambda (\tilde{u}-h') \delta_{0}(x)$}
 \psfrag{t^1}{$t^{1}$}
 \psfrag{t^2}{$t^{2}$}
 \psfrag{vn}{$v^{n}$}
 \psfrag{x_1/2^0}{$x_{1/2}^{0}$}
 \psfrag{x_1/2^1}{$x_{1/2}^{1}$}
\begin{figure}[htp]
\centering
 \includegraphics[width=\linewidth]{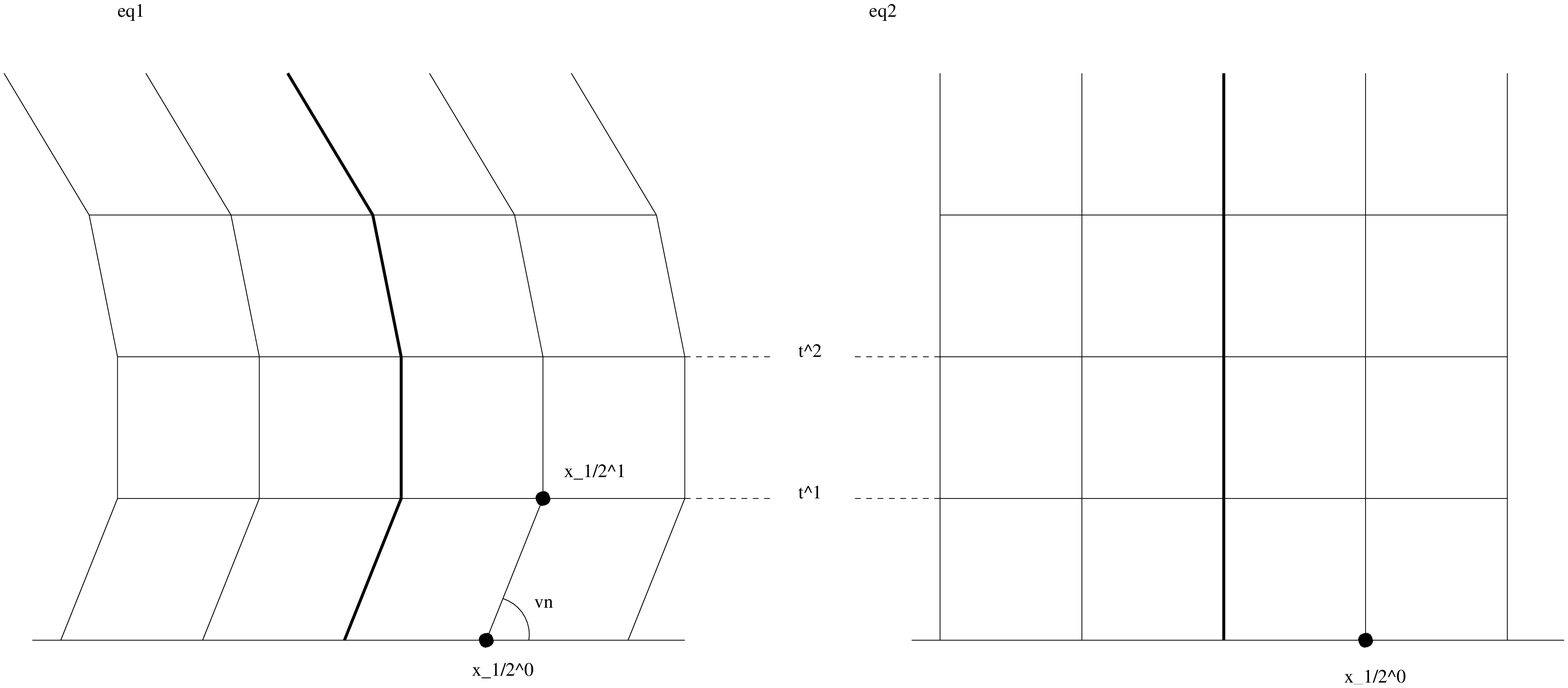}
 \caption[Two different ways to approximate the solution.]{To approximate the solution of~\eqref{eq:CauchyPb}, we can either use a mesh that follows the particle (on the left) or straighten the particle's trajectory and approximate the solution of~\eqref{eq:redressement}. In both case, the particle's trajectory is the bold line.} \label{F:POV}
\end{figure}
\end{psfrags}

The fluxes $f_{j-1/2}^{n}(v^{n})$ with $j \neq 1/2$ (or $f_{1/2}^{n \pm}(v^{n})$ if $j=1/2$) are strongly related to the fluxes $f_{j-1/2}^{v^{n},n}$: in~\eqref{eq:CoupledScheme}, $f_{j+1/2}^{n}(v^{n})$ is an approximation of
$$ \frac{1}{\Delta t} \int_{n \Delta t}^{(n+1) \Delta t}(f^{0}(u)-v^{n} u)(t,x_{j+1/2}^{n}+ v^{n} t)  dt $$
while in~\eqref{eq:CoupledSchemeR}, $f_{-1/2}^{v^{n},n}$ is an approximation of
$$ \frac{1}{\Delta t} \int_{n \Delta t}^{(n+1) \Delta t} f^{v^{n}}(\tilde{u})(t,x_{j+1/2}^{0}) dt $$

In the following we prove the convergence of Scheme~\eqref{eq:CoupledScheme} under a set of assumptions on the fluxes $f_{j+1/2}^{n}$, $f_{1/2,-}^{n}$ and $f_{j+1/2,+}^{n}$ and a Courant-Friedrichs-Lewy condition.  We restrict the study to two-points fluxes
 $$ f_{j+1/2}^{n}= g(u_{j}^{n}, u_{j+1}^{n}, v^{n}) \ \ \text{ and } \ \ \ f_{1/2, \pm}^{n}= g^{\pm}_{\lambda}(u_{j}^{n}, u_{j+1}^{n}, v^{n}). $$
 The assumptions on the flux $f_{j+1/2}^{n}$ away from the particle are the classical ones:
 \begin{itemize}
  \item consistency with the modified Burgers equation:
  \begin{equation} \label{gcons}
 \forall a \in \R, \ \forall v \in \R, \ g(a, a, v)= \frac{a^{2}}{2} -v a ,
\end{equation}
  \item monotonicity with respect to the first two arguments: 
\begin{equation} \label{gmono}
 \forall (a,b) \in \R^{2}, \ \forall v \in \R, \ \  \partial_{1} g(a,b,v) \geq 0 \ \ \text{ and } \ \ \partial_{2} g(a,b,v) \leq 0.
\end{equation}
  \item \inlineequation[gLip]{\text{ $g$ is locally Lipschitz-continuous;}} 
\end{itemize}
they ensure convergence of the scheme to an entropy solution of the Burgers equation away from the particle.

The assumptions on the fluxes around the particle are the following. We first have some consistency assumptions, which ensure that some particular solutions corresponding to a large enough part of the germ are exactly preserved by the numerical scheme. 
 We do not ask the flux to preserve the whole germ though, but only, in Section~\ref{S:ConvMax} with a \emph{maximal} part of the germ, and in Section~\ref{S:ConvLine}, with $\mathcal{G}_{\lambda}^{1}$. More precisely, the hypothesis on the fluxes $g^{\pm}_{\lambda}$ are:
\begin{itemize}
 \item consistency the part $\mathcal{G}_{\lambda}^{1}$ of the germ: 
\begin{equation} \label{eq:WBG1}
 \forall v \in \R, \ \forall (a,b) \in \mathcal{G}_{\lambda}^{1}, \ g^{-}_{\lambda}(a,b,v)= \frac{a^{2}}{2} - v a \ \ \text{ and } \ \ g^{+}_{\lambda}(a,b,v)= \frac{b^{2}}{2} -v b.
\end{equation}
In Section~\ref{S:ConvLine}, we make the stronger assumption that $g$ is consistent with a \emph{maximal} subset $\mathcal{H}_{\lambda}$ of $\mathcal{G}_{\lambda}$ (see Definition~\ref{def:MaxSub})
 \begin{equation} \label{eq:WBMax}
 \forall v \in \R, \ \forall (a,b) \in \mathcal{H}_{\lambda}(v), \ g^{-}_{\lambda}(a,b,v)= \frac{a^{2}}{2} - v a \ \ \text{ and } \ \ g^{+}_{\lambda}(a,b,v)= \frac{b^{2}}{2} -v b.
\end{equation}
\end{itemize}
Hypothesis~\eqref{eq:WBG1} will be used to prove $BV$ estimates on the fluid part $(u_{j}^{n})_{j \in \Z, n \in \N}$.
We also assume that
\begin{itemize}
 \item if the particle has the same velocity than the fluid, its velocity does not change:
\begin{equation} \label{eq:WBTip}
 \forall v \in \R, \ g^{-}_{\lambda}(v,v,v)= g^{+}_{\lambda}(v,v,v).
\end{equation}
\end{itemize}
 This hypothesis will be used to prove a $L^{\infty}$ bound on the particle velocity $(v^{n})_{n \in \N}$.
We add two classical conditions of regularity and monotonicity, also used to prove the $BV$ bound on $(u_{j}^{n})_{ j \in \Z, n \in \N}$. We assume that:
\begin{itemize}
 \item   \inlineequation[gpmLip]{\text{both $g^{-}_{\lambda}$ and $g^{+}_{\lambda}$ are locally Lipschitz-continuous; \hfill ~ }}
 
 \item $g^{-}_{\lambda}$ and $g^{+}_{\lambda}$ are nondecreasing with respect to their first arguments, and nonincreasing with respect to their second  \inlineequation[gpmmono]{\text{arguments.}}
\end{itemize}
 Just like in~\cite{AS12}, we need a dissipativity property to prove discrete entropy inequalities. Moreover, it will also be a key assumption to prove the bounds on the particle's velocity.
\begin{itemize} 
 \item \inlineequation[dissipativity]{\text{The function $g^{-}_{\lambda}-g^{+}_{\lambda}$ is nondecreasing with respect to its first two arguments.}} 
\end{itemize}
For this family of finite volume schemes, we are able to prove the following convergence theorem.
\begin{thm} \label{thm:conv}
 Consider a finite volume scheme of the form~\eqref{eq:CoupledScheme} that satisfies the set of hypothesis (\ref{gcons}--\ref{eq:WBG1}) and (\ref{eq:WBTip}--\ref{dissipativity}), and~\eqref{eq:WBMax} in Section~\ref{S:ConvLine}. Suppose that $u^{0}$ belongs to $BV(\R) \cap L^{1}(\R)$. Let us denote by $L$ the largest Lipschitz constant of $g$, $g^{+}$ and $g^{-}$ on the set $[m,M]^{2} \times [\underline{v}, \bar{v}]$, where
$$ 
\begin{cases}
 m= \min \{ \essinf_{\R^{-}} u^{0}- \lambda,\essinf_{\R^{+}} u^{0} \},  \\
 M= \max \{ \esssup_{\R^{-}} u^{0} ,\esssup_{\R^{+}} u^{0} + \lambda \}, \\
  \underline{v} = \min(m, v^{0}), \\
  \bar{v}= \max(M, v^{0}). 
\end{cases}
$$ 
Then, under the Courant-Friedrichs-Lewy condition
\begin{equation} \label{eq:CFL}
 L \mu \leq \frac{1}{2},
\end{equation}
the sequence $(u_{\Delta t})_{}$ converges in $L^{1}_{loc}(\R_{+} \times \R)$ toward $u$ and the sequence $(h_{\Delta t})_{}$ converges in $W^{1, \infty}_{loc}(\R_{+})$ toward $h$ when $\Delta t$ tends to $0$, where $(h,u)$ is the solution of~\eqref{eq:CauchyPb}.
\end{thm}
The next three Sections are devoted to the proof. In Section~\ref{S:Extract}, we prove bounds on the total variation of the fluid and on the acceleration of the particle, which permit us to extract converging subsequences. Then in Section~\ref{S:ConvMax}, we prove Theorem~\ref{thm:conv} under Hypothesis~\eqref{eq:WBMax}, which is sufficient to obtain a discrete version of~\eqref{eq:DefFluid1}. In Section~\ref{S:ConvLine}, we drop hypothesis~\eqref{eq:WBMax} and prove the convergence of the family of schemes such that
$$ 
\begin{cases}
 g^{-}_{\lambda}(a,b,v)&=g(a,b+\lambda,v), \\
 g^{+}_{\lambda}(a,b,v)&=g(a-\lambda,b,v),
\end{cases}
$$
where $g$ satisfies assuptions (\ref{gcons}--\ref{gLip}). This type of schemes was introduced in~\cite{AS12}. They only preserve the part $\mathcal{G}_{\lambda}^{1}$ of the germ, in the sense that if $(a,b)$ belongs to $\mathcal{G}_{\lambda}^{1}$, then
$$ g^{-}_{\lambda}(a,b,v)= f^{v}(a) \quad \text{ and } \quad g^{-}_{\lambda}(a,b,v)= f^{v}(b). $$
We recall that $\mathcal{G}_{\lambda}^{1}$ is not a maximal subset of $\mathcal{G}_{\lambda}(v)$. Under the set of assumptions specified above (except~\eqref{eq:WBMax}) we extend the proof of convergence of~\cite{AS12} to the fully coupled case~\eqref{eq:CauchyPb}.  

\section{A priori bounds}\label{S:Extract}
In the sequel we suppose that $u^{0}$ belongs to $L^{1}(\R) \cap BV(\R)$, that Hypothesis~\eqref{gcons}, \eqref{gmono} and~\eqref{gLip} on the flux $g$ are fulfilled, and that the monotonicity and regularity assumptions~\eqref{gpmmono} and~\eqref{gpmLip} on $g^{\pm}$ are verified. We will specify the consistency hypothesis on $g^{\pm}$ along the way. We first consider the uncoupled problem where $(v^{n})_{n \in \N}$ is fixed.

\begin{prop} \label{prop:FluidBound}
Let $u^{0}$ be in $BV(\R) \cap L^{1}(\R)$. Let $(v^{n})_{n \in \N}$ be given and $\underline{v}$ and $\bar{v}$ in $\R$ such that 
$$ \forall n \in \N, \ \underline{v} \leq v^{n} \leq \bar{v}. $$ 
Consider the finite volume scheme
\begin{equation*} 
\begin{cases}
 u_{j}^{n+1} &= u_{j}^{n}- \mu (g( u_{j}^{n}, u_{j+1}^{n}, v^{n}) - g(u_{j-1}^{n}, u_{j}^{n},v^{n}))  \text{ for } j \in \Z \setminus \{ 0,1 \}, \\
 u_{0}^{n+1} &= u_{0}^{n}- \mu (g^{-}_{\lambda}( u_{0}^{n}, u_{1}^{n}, v^{n})- g(u_{-1}^{n}, u_{0}^{n},v^{n}) ), \\
 u_{1}^{n+1} &= u_{1}^{n}- \mu (g( u_{1}^{n}, u_{2}^{n}, v^{n}) - g^{+}_{\lambda}( u_{0}^{n}, u_{1}^{n}, v^{n}) ) . 
\end{cases}
\end{equation*}
Suppose that the fluxes $g^{\pm}$ verify~\eqref{eq:WBG1} and that the CFL condition~\eqref{eq:CFL} holds. Then we have the following $L^{\infty}$ and $BV$ estimates in space on $u_{\Delta t}$, with $m$ and $M$ the constants of Theorem~\ref{thm:conv}:
\begin{equation} \label{eq:LinfBound}
  \forall n \geq 0, \forall j \in \Z, \ m \leq u_{j}^{n+1} \leq M
\end{equation}
and
\begin{equation} \label{eq:TVDBound}
\forall n \in \N, \ \sum_{j \in \Z} |u_{j}^{n}-u_{j-1}^{n}| \leq \sum_{j \in \Z} |u_{j}^{0}-u_{j-1}^{0}| + 2 \lambda.
\end{equation}
\end{prop}

\begin{proof} Due to the presence of the particle, the maximum and the total variation of the exact solution $u$ of~\eqref{eq:CauchyPb} can increase through time. For example if $u^{0}$ is constant equals to $0$ and if $v^{0}>\lambda$, then $||u(0^{+}, \cdot) ||_{L^{\infty}(\R)}=||u^{0} ||_{L^{\infty}(\R)} + \lambda $ and $||u(0^{+}, \cdot) ||_{BV(\R)}=||u^{0} ||_{BV(\R)} + 2 \lambda$ (see~\cite{LST08}, Lemma $5.7$). This prevents us for applying the LeRoux and Harten lemma (see ~\cite{H84} and~\cite{L77}) directly to $(u_{j}^{n})_{j \in \Z, \, n \in \N}$. Yet it can be applied to the sequence $(w_{j}^{n})_{j \in \Z, \, n \in \N}$ defined by
$$ w_{j}^{n} =
\begin{cases}
  u_{j}^{n} - \frac{\lambda}{2} & \text{if } j \leq 0, \\
 u_{j}^{n} + \frac{\lambda}{2} & \text{if } j \geq 1.
\end{cases}
$$
Let us prove that  there exists two families of real $(C_{j+1/2}^{n})_{j \in \Z, n \in \N}$ and $(D_{j+1/2}^{n})_{j \in \Z, n \in \N}$ such that  for all $j$ in $\Z$, for all $n$ in $\N$,
\begin{equation} \label{eq:w}
 w_{j}^{n+1}= w_{j}^{n} + C_{j+1/2}^{n}(w_{j+1}^{n}-w_{j}^{n}) - D_{j-1/2}^{n}(w_{j}^{n}-w_{j-1}^{n}), 
\end{equation}
and 
$$0 \leq 1-C_{j+1/2}^{n}- D_{j+1/2}^{n} \leq 1, \ 0 \leq C_{j+1/2}^{n} \leq 1 \ \text{ and }  0 \leq  D_{j+1/2}^{n} \leq 1. $$
In other words, $w_{j}^{n+1}$ writes as a convex combination of $w_{j-1}^{n}$, $w_{j}^{n}$ and $w_{j+1}^{n}$ and therefore, 
$$ \forall n \geq 0, \ \min_{k} w_{k}^{n} \leq w_{j}^{n+1} \leq \max_{k} w_{k}^{n}.$$
As a consequence, for all $n \in \N$ and  for $j \leq 0$,
$$ \min_{k} w_{k}^{0}+ \lambda/2 \leq u_{j}^{n} \leq \max_{k} w_{k}^{0}+ \lambda/2 $$
which rewrites
$$ \min \{ \cdots, u_{0}^{0}, u_{1}^{0}+ \lambda, \cdots  \} \leq u_{j}^{n} \leq  \max \{ \cdots, u_{0}^{0}, u_{1}^{0}+ \lambda, \cdots  \}.$$
Similarly, for all $n \in \N$ and for all $j \geq 1$,
$$ \min \{ \cdots, u_{0}^{0}-\lambda, u_{1}^{0}, \cdots  \} \leq u_{j}^{n} \leq  \max \{ \cdots, u_{0}^{0}-\lambda, u_{1}^{0}, \cdots  \},$$
hence the $L^{\infty}$ bound~\eqref{eq:LinfBound} is proven. Moreover, the LeRoux and Harten lemma yields
$$ \forall n \in \N, \  \sum_{j \in \Z} |w_{j}^{n+1}-w_{j-1}^{n+1}| \leq \sum_{j \in \Z} |w_{j}^{n}-w_{j-1}^{n}|, $$
and thus~\eqref{eq:TVDBound}.

Let us go back to the existence of $C_{j+1/2}^{n}$ and $D_{j-1/2}^{n}$.  In the sequel we denote by $|a,b|$ the interval $[\min(a,b), \max(a,b)]$. Suppose first that~\eqref{eq:w} holds for some $n \in \N$.  Then for every $j \leq -1$, there exists $\tilde{w}_{j-1/2}^{n} \in | w_{j-1}^{n}, w_{j}^{n} |$ and $\bar{w}_{j+1/2}^{n} \in | w_{j}^{n}, w_{j+1}^{n} |$
$$ 
\begin{aligned}
 w_{j}^{n+1} 	&= w_{j}^{n}- \mu \left( g(u_{j}^{n},u_{j+1}^{n}, v^{n}) - g(u_{j-1}^{n},u_{j}^{n}, v^{n}) \right) \\
 			&=w_{j}^{n}- \mu \left( g_{\lambda}\left(w_{j}^{n}+ \frac{\lambda}{2},w_{j+1}^{n}+ \frac{\lambda}{2}, v^{n} \right) - g_{\lambda}\left(w_{j-1}^{n}+ \frac{\lambda}{2},w_{j}^{n}+ \frac{\lambda}{2}, v^{n}\right) \right) \\
			&= w_{j}^{n} - \mu \left( \partial_{1}g_{\lambda}\left( \tilde{w}_{j-1/2}^{n}+ \frac{\lambda}{2}, w_{j+1}^{n}+ \frac{\lambda}{2}, v^{n}\right) (w_{j}^{n}-w_{j-1}^{n})  \right. \\
			&\left. \ \ \ \ \ \ \ \ \ \ \ \ \ \ \ \ \ \  \ \ +  \partial_{2} g_{\lambda}\left(w_{j-1}^{n}+ \frac{\lambda}{2}, \bar{w}_{j+1/2}^{n}+ \frac{\lambda}{2}, v^{n} \right)(w_{j+1}-w_{j}) \right)
\end{aligned}
$$
Both triplets $\left( \tilde{w}_{j-1/2}^{n}+ \frac{\lambda}{2}, w_{j+1}^{n}+ \frac{\lambda}{2}, v^{n}\right)$ and $\left(w_{j-1}^{n}+ \frac{\lambda}{2}, \bar{w}_{j+1/2}^{n}+ \frac{\lambda}{2}, v^{n} \right)$ belong to $[m,M]^{2} \times [\underline{v}, \bar{v}]$. The CFL condition~\eqref{eq:CFL}, and the fact that $\partial_{1} g \geq 0$ and $\partial_{2} g \leq 0$,  yield~\eqref{eq:w} with
$$ 
\begin{cases}
 D_{j-1/2}^{n} &= \mu \partial_{1}g_{\lambda}\left( \tilde{w}_{j-1/2}^{n}+ \frac{\lambda}{2}, w_{j+1}^{n}+ \frac{\lambda}{2}, v^{n}\right), \\
 C_{j+1/2}^{n} &= - \mu\partial_{2}g_{\lambda}\left(w_{j-1}^{n}+ \frac{\lambda}{2}, \bar{w}_{j+1/2}^{n}+ \frac{\lambda}{2}, v^{n} \right). 
\end{cases}
$$
The case $j \geq 2$ can be treated in the exact same way. We now turn to the trickier case $j=0$. 
The facts that $g^{-}_{\lambda}$ is consistent with $\mathcal{G}_{\lambda}^{1}$ and that $g$ is consistent (Hypothesis~\eqref{eq:WBG1} and~\eqref{gcons}) imply that
$$ g^{-}_{\lambda}\left(w_{0}^{n}+ \frac{\lambda}{2},w_{0}^{n}- \frac{\lambda}{2}, v^{n} \right)= g_{\lambda}\left(w_{0}^{n}+ \frac{\lambda}{2},w_{0}^{n}+ \frac{\lambda}{2},v^{n} \right), $$ 
which allows us to write 
$$ 
\begin{aligned}
 w_{0}^{n+1} 	&= w_{0}^{n}- \mu \left( g^{-}_{\lambda}(u_{0}^{n},u_{1}^{n}, v^{n}) - g (u_{-1}^{n},u_{0}^{n}, v^{n}) \right) \\
 			&=w_{0}^{n}- \mu \left( g^{-}_{\lambda}\left(w_{0}^{n}+ \frac{\lambda}{2},w_{1}^{n}- \frac{\lambda}{2}, v^{n} \right) - g_{\lambda}\left(w_{-1}^{n}+ \frac{\lambda}{2},w_{0}^{n}+ \frac{\lambda}{2}, v^{n}\right) \right) \\
			&= w_{0}^{n} - \mu \left(   g^{-}_{\lambda}\left(w_{0}^{n}+ \frac{\lambda}{2},w_{1}^{n}- \frac{\lambda}{2}, v^{n} \right) -   g^{-}_{\lambda}\left(w_{0}^{n}+ \frac{\lambda}{2},w_{0}^{n}- \frac{\lambda}{2}, v^{n} \right) \right. \\
			&\left. \ \ \ \ \ \ \ \ \ \ \ \ \ \ \ \ \ \  \ \ +g_{\lambda}\left(w_{0}^{n}+ \frac{\lambda}{2},w_{0}^{n}+ \frac{\lambda}{2}, v^{n} \right) -g_{\lambda}\left(w_{-1}^{n}+ \frac{\lambda}{2},w_{0}^{n}+ \frac{\lambda}{2}, v^{n}\right)\right) 
\end{aligned}
$$
Thus, there exists  $\tilde{w}_{-1/2}^{n} \in | w_{-1}^{n}, w_{0}^{n} |$ and $\bar{w}_{1/2}^{n} \in | w_{0}^{n}, w_{1}^{n} |$ such that
$$ 
\begin{aligned}
 w_{0}^{n+1} 	&= w_{0}^{n} - \mu \left(  \partial_{2} g^{-}_{\lambda}\left(w_{0}^{n}+ \frac{\lambda}{2},\bar{w}_{1/2}^{n}- \frac{\lambda}{2}, v^{n} \right) (w_{1}^{n}-w_{0}^{n}) \right. \\
			&\left. \ \ \ \ \ \ \ \ \ \ \ \ \ \ \ \ \ \  \ \ +\partial_{1} g_{\lambda}\left(\tilde{w}_{-1/2}^{n}+ \frac{\lambda}{2}, w_{0}^{n}+ \frac{\lambda}{2}, v^{n} \right) (w_{0}^{n}- w_{-1}^{n})\right)
\end{aligned}
$$
Once again, both triplets  $\left(w_{0}^{n}+ \frac{\lambda}{2},\bar{w}_{1/2}^{n}- \frac{\lambda}{2}, v^{n} \right)$ and $\left(w_{0}^{n}+ \frac{\lambda}{2},\tilde{w}_{-1/2}^{n}+ \frac{\lambda}{2}, v^{n} \right)$ belong to $[m,M]^{2} \times [\underline{v}, \bar{v}]$. The monotonicity on $g$ and $g^{-}_{\lambda}$ allow to conclude with
$$ 
\begin{cases}
 D_{-1/2}^{n} &= \mu \partial_{1} g_{\lambda}\left(\tilde{w}_{-1/2}^{n}+ \frac{\lambda}{2}, w_{0}^{n}+ \frac{\lambda}{2}, v^{n} \right),\\
 C_{1/2}^{n} &= - \mu \partial_{2} g^{-}_{\lambda}\left(w_{0}^{n}+ \frac{\lambda}{2},\bar{w}_{1/2}^{n}- \frac{\lambda}{2}, v^{n} \right).
\end{cases}
$$
The case $j=1$ can be treated in the exact same way, using the consistency assumption
$$ g^{+}_{\lambda}\left(w_{1}^{n}+ \frac{\lambda}{2},w_{1}^{n}- \frac{\lambda}{2}, v^{n} \right) = g_{\lambda}\left(w_{1}^{n}- \frac{\lambda}{2},w_{1}^{n}- \frac{\lambda}{2},v^{n} \right). $$
\end{proof}

We now turn to the case where the particle's velocity is updated from time to time, and focus on the estimates on the velocity and acceleration of the particle.
\begin{prop} \label{prop:PartBound}
Suppose that the fluxes $g^{\pm}$ verify~\eqref{eq:WBTip}, \eqref{dissipativity} and~\eqref{eq:CFL} that the time step verifies
\begin{equation}  \label{eq:MassCond}
 \frac{4L}{m_{p} } \Delta t \leq 1,
\end{equation}
Then,  the sequence $(u_{j}^{n})_{ j \in \Z, n \in \N}$ (defined by ~\eqref{eq:CoupledScheme}) verifies Estimates~\eqref{eq:LinfBound} and~\eqref{eq:TVDBound}, while $(v^{n})_{n \in \N}$ verifies the following estimates:
\begin{equation} \label{eq:vitBound}	
 \forall n \in \N, \ \ \ \underline{v} \leq v^{n} \leq \bar{v},
\end{equation}
and
\begin{equation}  \label{eq:accBound}	
 \forall n \in \N, \ \ \ \left| \frac{ v^{n+1}-v^{n}}{\Delta t} \right| \leq \frac{2L}{m_{p}} (||u^{0}||_{\infty}+ \lambda+ ||v||_{\infty}). 
\end{equation}
 The constants $\bar{v}$ and $\underline{v}$ are defined in Theorem~\ref{thm:conv}.
\end{prop}

\begin{proof}
 We proceed by induction. Let us first remark that if the estimate~\eqref{eq:vitBound} on $v^{n}$ is fulfilled at time $t^{n}$, the proof of Proposition~\ref{prop:FluidBound} yields the $L^{\infty}$ and $BV$ estimates on $(u_{j}^{n+1})_{j \in \Z}$. Therefore, we focus on the estimate on $v^{n+1}$. Using Hypothesis~\eqref{eq:WBTip}, we introduce the null quantity $g^{-}_{\lambda}(v^{n}, v^{n}, v^{n})-g^{+}_{\lambda}(v^{n}, v^{n}, v^{n})$ and write
 $$ 
\begin{aligned}
 v^{n+1} &= v^{n} + \frac{\Delta t}{m_{p}} ( g^{-}_{\lambda}(u_{0}^{n}, u_{1}^{n}, v^{n}) - g^{+}_{\lambda}(u_{0}^{n}, u_{1}^{n}, v^{n}) ) \\
 	&=  v^{n} + \frac{\Delta t}{m_{p}} \left( \int_{0}^{1} \partial_{s} ( g^{-}_{\lambda}(v^{n} + s(u_{0}^{n}-v^{n}), v^{n} + s(u_{1}^{n}-v^{n}), v^{n}) ) ds \right. \\
	&  \ \ \ \ \ \ \ \ \ \  -\left.  \int_{0}^{1} \partial_{s} ( g^{+}_{\lambda}(v^{n} + s(u_{0}^{n}-v^{n}), v^{n} + s(u_{1}^{n}-v^{n}), v^{n})  ) ds  \right),
\end{aligned}
$$
and we obtain
\begin{equation} \label{eq:IntForm}
 \begin{aligned}
 v^{n+1} &=  v^{n} + \frac{\Delta t}{m_{p}} \left( \int_{0}^{1} (u_{0}^{n}-v^{n}) \partial_{1} (g^{-}- g^{+})(v^{n} + s(u_{0}^{n}-v^{n}), v^{n} + s(u_{1}^{n}-v^{n}), v^{n}) ds \right. \\
	&  \ \ \ \ \ \ \ \ \ \ + \left. \int_{0}^{1} (u_{1}^{n}-v^{n}) \partial_{2} (g^{-}- g^{+})(v^{n} + s(u_{0}^{n}-v^{n}), v^{n} + s(u_{1}^{n}-v^{n}), v^{n}) ds \right) 
\end{aligned}
\end{equation}
	
Suppose now that $v^{n} \leq \min(u_{0}^{n}, u_{1}^{n})$. Then both $ (u_{1}^{n}-v^{n})$ and $ (u_{0}^{n}-v^{n})$ are nonnegative. Moreover, the dissipativity assumption~\eqref{dissipativity} implies that $\partial_{1} (g^{-}- g^{+})$ and $\partial_{2} (g^{-}- g^{+})$ are also nonnegative. Hence we have $v^{n+1} \geq v^{n}$ and  Hypothesis~\eqref{eq:MassCond} yields
$$  
\begin{aligned}
  v^{n+1} & \leq v^{n} + 2L \frac{\Delta t}{m_{p}} (u_{0}^{n}-v^{n} + u_{1}^{n}-v^{n}) \\
  	& \leq \left( 1-\frac{4L \Delta t}{m_{p}} \right) v^{n} + \frac{4L \Delta t}{m_{p} } \max(u_{0}^{n}, u_{1}^{n}) \\
	& \leq \bar{v}. 
\end{aligned}
$$
We now treat the case $u_{0}^{n} \leq v^{n} \leq u_{1}^{n}$. The only difference is that $u_{0}^{n}-v^{n}$ is now negative. The integral form~\eqref{eq:IntForm} of $v^{n+1}$ and Hypothesis~\eqref{eq:vitBound} yield 
$$ \underline{v} \leq v^{n}- 2L \frac{\Delta t}{m_{p}} (v^{n}-u_{0}^{n}) \leq v^{n+1} \leq  v^{n}+ 2L \frac{\Delta t}{m_{p}} (u_{1}^{n}-v^{n}) \leq \bar{v}. $$ 
Once the $L^{\infty}$ bounds on $(u^{n}_{j})_{j \in \Z, n \in \N}$ and $(v^{n})_{n \in \N}$ are proven, the bound of the particle's acceleration~\eqref{eq:accBound} is an easy consequence of the integral form of $v^{n+1}$.
\end{proof}

\begin{rem}
Condition~\eqref{eq:MassCond} is fulfilled for small enough $\Delta t$. Thus it is not a restriction to prove the convergence of the scheme. However from the numerical point of view, one has to check Condition~\eqref{eq:MassCond} in addition to the CFL condition~\eqref{eq:CFL}. This restriction is severe if the particle is very light. It is possible, at the cost of solving a nonlinear system, to use an implicit version of Scheme~\eqref{eq:CoupledScheme} for the particle's velocity, i.e.
 \begin{equation*} 
 \begin{cases}
 u_{j}^{n+1} &= u_{j}^{n}- \mu (f_{j+1/2}^{n}(v^{n+1}) - f_{j-1/2}^{n} (v^{n+1}))  \text{ for $j \in \Z, j \notin \{0, 1\}$}, \\
 u_{0}^{n+1} &= u_{0}^{n}- \mu (f_{1/2,-}^{n}(v^{n+1}) - f_{-1/2}^{n}(v^{n+1})), \\
 u_{1}^{n+1} &= u_{1}^{n}- \mu (f_{3/2}^{n}(v^{n+1}) - f_{1/2,+}^{n}(v^{n+1})), \\
 v^{n+1} & = v^{n} + \frac{\Delta t}{m_{p}} (f_{1/2,-}^{n}(v^{n+1})-f_{1/2,+}^{n}(v^{n+1})), \\
 x_{j}^{n+1}& =x_{j}^{n}+ v^{n}\Delta t.
\end{cases}
\end{equation*}
In that case, we obtain Bounds~\eqref{eq:vitBound} and~\eqref{eq:accBound} without Constraint~\eqref{eq:MassCond} on the time step. The proof is exactly the same than the one of Proposition~\ref{prop:PartBound}. For example in the case where $v^{n+1} \leq \min(u_{0}^{n}, u_{1}^{n})$, we obtain 
$$ v^{n+1}  \leq v^{n} + 2L \frac{\Delta t}{m_{p}} (u_{0}^{n}-v^{n+1} + u_{1}^{n}-v^{n+1}), $$ and thus, without any constraint on $\Delta t$ other than~\eqref{eq:CFL},
\begin{equation*}
 v^{n+1} \leq \frac{v^{n}+ 2L \frac{\Delta t}{m_{p}}(u_{0}^{n}+u_{1}^{n})}{1+4L \frac{\Delta t}{m_{p}}} \leq \bar{v}. \qedhere
\end{equation*} 
\end{rem}

We are now in position to extract converging subsequences of $(u_{\Delta})$ and $(h_{\Delta})$ (defined in~\eqref{def:udelta} and~\eqref{def:hdelta}). In Section~\ref{S:ConvMax}, we will prove that their limits are solutions of the Cauchy problem~\eqref{eq:CauchyPb} for the fully coupled problem.
\begin{prop} \label{thm:Extraction} Assume that $u^{0}$ belongs to $BV(\R) \cap L^{1}(\R)$, and that Hypothesis~(\ref{gcons}-\ref{eq:WBG1}) and~(\ref{eq:WBTip}-\ref{dissipativity}) are verified. Moreover, suppose that the CFL condition~\eqref{eq:CFL} holds. Then there exists $u$ in $BV_{loc}(\R_{+} \times \R)$ and $h$ in $W^{2,\infty}_{loc}(\R_{+})$ such that, up to a subsequence, the sequence $(u_{\Delta t})_{} $ converges in $L^{1}_{loc}(\R_{+} \times \R)$ toward $u$ and the sequence $(h_{\Delta t})_{} $ converges in $W^{1, \infty}_{loc}(\R)$ toward $h$ as $\Delta t$ tends to $0$. 
\end{prop}
\begin{proof} Let us first fix a time $T>0$ and a constant $A>0$ and prove the convergence in $L^{1}([0,T] \times [-A,A])$ and $W^{1, \infty}([0,T]) $. By Proposition~\ref{prop:FluidBound}, we can use Helly's theorem to prove the convergence in $L^{1}([0,T] \times [-A,A])$ of $(u_{\Delta t})_{}$, toward a function $u$ in $BV(([0,T] \times [-A,A]))$. Similarly Proposition~\ref{prop:PartBound} allows us to apply Arzelà-Ascoli's theorem to prove convergence in $W^{1, \infty}([0,T])$ of $(h_{\Delta t})_{}$ to a function $h$ belonging to $W^{2, \infty}([0,T])$. The result is extended to the whole time-space $\R_{+} \times \R$ thanks to the Cantor diagonal extraction argument.
\end{proof}
\begin{rem} \label{rem:convc}	
Up to the same subsequence, $(v_{\Delta t})_{}$ converges toward $h'$ in $L^{1}_{loc}$. Moreover
 the sequence of functions $(c_{\Delta t})_{}$ defined by
$$
c_{\Delta t} (t,x)= 
\begin{cases}
 c_{-} & \text{ if } t< h_{\Delta x}(t), \\
 c_{+} & \text{ if } t> h_{\Delta x}(t), 
\end{cases}
$$
converges in $L^{1}_{loc}$ toward
$$
c (t,x)= 
\begin{cases}
 c_{-} & \text{ if } t< h(t),\\
 c_{+} & \text{ if } t> h(t).
\end{cases}
$$
Indeed, we have
 $$ \int_{-A}^{A} \int_{0}^{T} |c_{\Delta t} (t,x)-c(t,x)| dt dx \leq |c_{+}-c_{-}| \int_{0}^{T} |h_{\Delta t}(t) - h(t) | dt \leq 2 L T \Delta t.$$
\end{rem}

\section{Convergence of schemes consistent with a maximal part of the germ} \label{S:ConvMax}
For now on, we suppose that all the hypotheses of Proposition~\ref{thm:Extraction}  are fulfilled, and that both Conditions~\eqref{eq:CFL} and~\eqref{eq:MassCond} are verified. The aim of this section is to prove Theorem~\ref{thm:conv}. To that purpose, we prove that under Condition~\eqref{eq:WBMax}, which states that the fluxes $g^{\pm}_{\lambda}$ around the particle are consistent with a maximal subset $\mathcal{H}_{\lambda}$ of the germ (see Definition~\ref{def:MaxSub}), the limit $(u,h)$ of the scheme is the solution of~\eqref{eq:CauchyPb}. 

The fact that the Cauchy problem~\eqref{eq:CauchyPb} is well posed in $BV(\R)$ is proven in~\cite{ALST13}. Once we know that Scheme~\eqref{eq:CoupledScheme} converges toward a solution of~\eqref{eq:CauchyPb}, the uniqueness of the solution yields that the whole sequence $(u_{\Delta t}, h_{\Delta t})$ converges. Theorem~\ref{thm:conv} gives a different way to prove the existence of a solution (but not the uniqueness).

\subsection{Convergence of the fluid's part}
The aim of this subsection is to prove that the limit $u$ of $(u_{\Delta t})$ verifies~\eqref{eq:DefFluid1}. We prove in Proposition~\ref{prop:EI2} that $(u_{j}^{n})_{j \in \Z, n \in \N}$ verifies a discrete version of~\eqref{eq:DefFluid1}. In the sequel, for all reals number $a$ and $b$ we denote by
$$ a \top b= \max(a,b) \ \text{ and by } \ a \bot b= \min(a,b). $$ 
In the following proposition, we establish a discrete entropy inequality.
\begin{prop} \label{prop:EntIneq2Prel} 
Assume that Hypothesis~(\ref{gcons}-\ref{dissipativity}) hold (included~\eqref{eq:WBMax}) and that the CFL condition~\eqref{eq:CFL} is fulfilled. Then for all $(c_{-}, c_{+})$ in $\R^{2}$, there exists  a constant $A$, depending only on $\lambda$, $||u^{0}||_{\infty}$, $|| v ||_{\infty}$ and $(c_{-},c_{+})$, such that for all $j\in \Z$, for all $n \in \N$, the following inequality holds:
 \begin{equation} \label{eq:DIE}
 \frac{|u_{j}^{n+1}- c_{j}|-|u_{j}^{n}- c_{j}|}{\Delta t} + \frac{G_{j+1/2,-}^{n}-G_{j-1/2,+}^{n}}{\Delta x} \leq \eps_{j} \frac{A}{\Delta x} \dist_{1}((c_{-},c_{+}), \mathcal{H}_{\lambda}(v^{n})),
\end{equation}
where 
$$ \forall j \neq 0, \ G_{j+1/2,-}^{n}= G_{j+1/2,+}^{n}= G_{j+1/2}^{n}, $$
with
$$G_{j+1/2}^{n}= g(u_{j}^{n} \top c_{j}, u_{j+1}^{n} \top c_{j+1}, v^{n})- g(u_{j}^{n} \bot c_{j}, u_{j+1}^{n} \bot c_{j+1}, v^{n}), $$
$$ G_{1/2, \pm}^{n }= g^{\pm}_{\lambda}(u_{0}^{n} \top c_{0}, u_{1}^{n} \top c_{1}, v^{n})- g^{\pm}_{\lambda}(u_{0}^{n} \bot c_{0}, u_{1}^{n} \bot c_{1}, v^{n}), $$
$$ c_{j}= 
\begin{cases}
 c_{-} & \text{ if } j \leq 0, \\
 c_{+} & \text{ if } j \geq 1,
\end{cases}
\ \ \ \text{and} \ \ \
\eps_{j}= 
\begin{cases}
 1 & \text{ if } j \in \{ 0, 1 \}, \\
 0 & \text{ otherwise. } 
\end{cases}
$$
\end{prop} 

\begin{proof}
We follow the guidelines of proofs of classical entropy inequalities. They rely on the identity
$$ |u_{j}^{n+1}-c_{j}|= u_{j}^{n+1} \top c_{j} - u_{j}^{n+1} \bot c_{j}. $$
For $j \in \Z \setminus \{0,1\}$, we use the condensed notation $u_{j}^{n+1}= H(u_{j-1}^{n}, u_{j}^{n}, u_{j+1}^{n}, v^{n})$. Hypothesis~\eqref{gmono} on the monotonicity of the fluxes and the CFL condition~\eqref{eq:CFL} ensure that for every $v$, $H$ is increasing with respect to its first three arguments. Moreover if $j \in \Z \setminus \{0,1\}$, $c_{j-1}=c_{j}=c_{j+1}$ and we use the consistency of the flux away from the particle~\eqref{gcons} to write $c_{j}= H(c_{j-1},c_{j},c_{j+1}, v^{n})$. It follows that
$$ 
\begin{aligned}
 u_{j}^{n+1} \top c_{j} 	&= H(u_{j-1}^{n}, u_{j}^{n}, u_{j+1}^{n}, v^{n}) \top H(c_{j-1}, c_{j}, c_{j+1}, v^{n}) \\
 					&\leq H(u_{j-1}^{n} \top c_{j-1}, u_{j}^{n}\top c_{j}, u_{j+1}^{n}\top c_{j+1}, v^{n}) \\
 u_{j}^{n+1} \bot c_{j} 	&= H(u_{j-1}^{n}, u_{j}^{n}, u_{j+1}^{n}, v^{n}) \bot H(c_{j-1}, c_{j}, c_{j+1}, v^{n}) \\
 					&\geq H(u_{j-1}^{n} \bot c_{j-1}, u_{j}^{n}\bot c_{j}, u_{j+1}^{n}\bot c_{j+1}, v^{n}) 
\end{aligned}
$$
and that
$$ 
\begin{aligned}
 |u_{j}^{n+1}-c_{j}| 	& \leq H(u_{j-1}^{n} \top c_{j-1}, u_{j}^{n}\top c_{j}, u_{j+1}^{n}\top c_{j+1}, v^{n}) - H(u_{j-1}^{n} \bot c_{j-1}, u_{j}^{n}\bot c_{j}, u_{j+1}^{n}\bot c_{j+1}, v^{n}) \\
 				& \leq u_{j}^{n}\top c_{j} - u_{j}^{n}\bot c_{j} -\mu ( G_{j+1/2}^{n}- G_{j-1/2}^{n} ) \\
				& \leq  |u_{j}^{n}-c_{j}| -\mu ( G_{j+1/2}^{n}- G_{j-1/2}^{n} )
\end{aligned}
$$
Let us now focus on the more complicated case $j=0$ (the case $j=1$ can be treated in the exact same way).  We denote by $(\tilde{c}_{0}^{n}, \tilde{c}_{1}^{n})$ a projection of $(c_{-}, c_{+})=(c_{0}^{n}, c_{1}^{n})$ on $\mathcal{H}_{\lambda}(v^{n})$ for the $L^{1}$-norm, and by $(\tilde{c}_{j}^{n})_{j \in \Z, n \in \N}$ and $(\tilde{G}_{j+1/2}^{n})_{j \in \Z, n \in \N}$ the analogues of $(c_{j})_{j \in \Z}$ and $(G_{j+1/2}^{n})_{j \in \Z, n \in \N}$ constructed with $\tilde{c}$:
$$ \forall j \neq 0, \ \tilde{G}_{j+1/2,-}^{n}= \tilde{G}_{j+1/2,+}^{n}=\tilde{G}_{j+1/2}^{n}= g(u_{j}^{n} \top \tilde{c}_{j}, u_{j+1}^{n} \top \tilde{c}_{j+1}, v^{n})- g(u_{j}^{n} \bot \tilde{c}_{j}, u_{j+1}^{n} \bot \tilde{c}_{j+1}, v^{n}), $$
$$ \tilde{G}_{1/2, \pm}^{n }= g^{\pm}_{\lambda}(u_{0}^{n} \top \tilde{c}_{0}, u_{1}^{n} \top \tilde{c}_{1}, v^{n})- g^{\pm}_{\lambda}(u_{0}^{n} \bot \tilde{c}_{0}, u_{1}^{n} \bot \tilde{c}_{1}, v^{n}). $$
Let us first remark that
$$ 
\begin{aligned}
 |u_{0}^{n+1}- c_{0}|-|u_{0}^{n}- c_{0}|	& \leq |u_{0}^{n+1}- \tilde{c}_{0}^{n}|+ | \tilde{c}_{0}^{n} - c_{0} |  -  \big| |u_{0}^{n}- \tilde{c}_{-}^{n}| - |  \tilde{c}_{0}^{n}-c_{0}|  \big| \\
 						& \leq |u_{0}^{n+1}- \tilde{c}_{0}^{n}| - |u_{0}^{n}- \tilde{c}_{-}^{n}| + 2 |  \tilde{c}_{0}^{n}-c_{0}| .
\end{aligned}
$$
Thus we have
$$ 
\begin{aligned}
& \frac{|u_{0}^{n+1}- c_{0}|-|u_{0}^{n}- c_{0}|}{\Delta t} + \frac{G_{1/2,-}^{n}-G_{-1/2}^{n}}{\Delta x} \\
& \ \ \ \ \ \ \ \ \ \ \ \ \leq \frac{|u_{0}^{n+1}- \tilde{c}_{0}^{n}|-|u_{0}^{n}- \tilde{c}_{0}^{n}|}{\Delta t} + \frac{G_{1/2,-}^{n}-G_{-1/2}^{n}}{\Delta x} + \frac{2}{\Delta t} \dist_{1}((c_{-},c_{+}), \mathcal{H}_{\lambda}(v^{n})) \\
& \ \ \ \ \ \ \ \ \ \ \ \ \leq \frac{G_{1/2,-}^{n}-G_{-1/2}^{n}}{\Delta x} - \frac{\tilde{G}_{1/2,-}^{n}-\tilde{G}_{-1/2}^{n}}{\Delta x} + \frac{2}{\Delta t} \dist_{1}((c_{-},c_{+}), \mathcal{H}_{\lambda}(v^{n}) ).
\end{aligned}
$$
Indeed, as $(\tilde{c}_{0}^{n}, \tilde{c}_{1}^{n})$ belongs to $\mathcal{H}_{\lambda}(v^{n})$, Hypothesis~\eqref{eq:WBMax} yields that $\tilde{c}_{0}^{n}= H_{\lambda}( \tilde{c}_{-1}^{n}, \tilde{c}_{0}^{n}, \tilde{c}_{1}^{n}, v^{n})$, and we obtain as before
$$ 
 |u_{0}^{n+1}-\tilde{c}_{j}| 	 \leq  |u_{j}^{n}-\tilde{c}_{j}| -\mu ( \tilde{G}_{j+1/2}^{n}- \tilde{G}_{j-1/2}^{n} ).
$$
We now attempt to bound
$$ 
\begin{aligned}
 G_{1/2,-}^{n}- \tilde{G}_{1/2,-}^{n} = g^{-}_{\lambda} &(u_{0}^{n} \top c_{0}, u_{1}^{n} \top c_{1}, v^{n}) - g^{-}_{\lambda}(u_{0}^{n} \bot c_{0}, u_{1}^{n} \bot c_{1}, v^{n}) \\
 			& - g^{-}_{\lambda}(u_{0}^{n} \top \tilde{c}_{0}^{n}, u_{1}^{n} \top \tilde{c}_{1}^{n}, v^{n}) + g^{-}_{\lambda}(u_{0}^{n} \bot \tilde{c}_{0}^{n}, u_{1}^{n} \bot \tilde{c}_{1}^{n}, v^{n}).
\end{aligned}
$$
As $(v^{n})_{n \in \Z}$ is bounded (Proposition~\ref{prop:PartBound}), the maximum and minimum over $n$ of $\tilde{c}_{\pm}^{n}$  is a bounded function of $(c_{-},c_{+})$ and $||v||_{\infty}$. Thus the set
$$[\min(m,c_{-},c_{+}, \tilde{c}_{-}^{n}, \tilde{c}_{+}^{n}), \max(M,c_{-},c_{+}, \tilde{c}_{-}^{n}, \tilde{c}_{+}^{n})]^{2} \times [\underline{v}, \bar{v}]. $$
is compact. Therefore, with $L_{c}$  the Lipschitz constant of $g^{-}_{\lambda}$ over this set, we have
$$ 
\begin{aligned}
 |  g^{-}_{\lambda} &(u_{0}^{n} \top c_{0}, u_{1}^{n} \top c_{1}, v^{n}) - g^{-}_{\lambda}(u_{0}^{n} \top \tilde{c}_{0}^{n}, u_{1}^{n} \top \tilde{c}_{1}^{n}, v^{n}) | \\
 	& \leq |  g^{-}_{\lambda}(u_{0}^{n} \top c_{0}, u_{1}^{n} \top c_{1}, v^{n}) - g^{-}_{\lambda}(u_{0}^{n} \top \tilde{c}_{0}^{n}, u_{1}^{n} \top c_{1}, v^{n}) |  \\
	& \qquad +  |  g^{-}_{\lambda}(u_{0}^{n} \top \tilde{c}_{0}^{n}, u_{1}^{n} \top c_{1}, v^{n}) - g^{-}_{\lambda}(u_{0}^{n} \top \tilde{c}_{0}^{n}, u_{1}^{n} \top \tilde{c}_{1}^{n}, v^{n}) |  \\
	&\leq L_{c} \dist_{1}((c_{-},c_{+}), \mathcal{H}_{\lambda}(v^{n})),
\end{aligned}
$$
and similarly
$$ |  g^{-}_{\lambda}(u_{0}^{n} \bot c_{0}, u_{1}^{n} \bot c_{1}, v^{n}) - g^{-}_{\lambda}(u_{0}^{n} \top \tilde{c}_{0}^{n}, u_{1}^{n} \top \tilde{c}_{1}^{n}, v^{n}) | \leq L_{c} \dist_{1}((c_{-},c_{+}), \mathcal{H}_{\lambda}(v^{n})), $$
which concludes the proof with $ A= 2 L_{c} + \frac{2 \Delta x}{\Delta t}$.
\end{proof}

We are now in position to obtain a discrete version of~\eqref{eq:DefFluid1}.
\begin{prop} \label{prop:EI2}
 Let $(\varphi_{j}^{n})_{j \in \Z, n \in \N}$ be a compactly supported sequence of nonnegative reals. If~\eqref{eq:DIE} holds for all $n$ in $\N$ and $j$ in $\Z$, then
 \begin{equation} \label{eq:EI2}
 \begin{aligned}
 \Delta t &\Delta x \sum_{j \in \Z, n \in \N} |u_{j}^{n+1}-c_{j}| \frac{\varphi_{j}^{n+1}-\varphi_{j}^{n}}{\Delta t} + \Delta x \sum_{i \in \Z} |u_{j}^{0}-c_{j}| \varphi_{j}^{0} + \Delta t \Delta x \sum_{j \in \Z^{*}, n \in \N} G_{j+1/2}^{n} \frac{\varphi_{j+1}^{n}-\varphi_{j}^{n}}{\Delta x}\\
 	&  + \Delta t \Delta x \sum_{ n \in \N} G_{j+1/2,+}^{n} \frac{\varphi_{1}^{n}-\varphi_{0}^{n}}{\Delta x} \geq -A \Delta t \sum_{n \in \N} \dist_{1}(c, \mathcal{H}_{\lambda}(v^{n})) (\varphi^{n}_{0}+\varphi^{n}_{1}).
\end{aligned}
\end{equation}
\end{prop}

\begin{proof}
Classically, the starting point is to multiply Equation~\eqref{eq:DIE} by $\varphi_{j}^{n}$ and to sum over $j \in \Z$ and $n \in \N$. Then the different terms are rearranged to bring out discrete time and space derivatives of $\varphi$. However, this is not straightforward around the particle, because two different fluxes are used on its left and on its right. 
The first term of~\eqref{eq:DIE} yields
 $$ 
\begin{aligned}
\sum_{j \in \Z, n \in \N} \frac{|u_{j}^{n+1}- c_{j}|-|u_{j}^{n}- c_{j}|}{\Delta t} \varphi_{j}^{n} = \sum_{j \in \Z, n \in \N} |u_{j}^{n+1}-c_{j}| \frac{\varphi_{j}^{n}- \varphi_{j}^{n+1}}{\Delta t} - \frac{1}{\Delta t} \sum_{j \in \Z} |u_{j}^{0}-c_{j}| \varphi_{j}^{0},
\end{aligned}
$$
and the second term yields
$$ 
\begin{aligned}
\sum_{j \in \Z, n \in \N} \frac{G_{j+1/2,-}^{n}-G_{j-1/2,+}^{n}}{\Delta x} \varphi_{j}^{n} &= \sum_{j \in \Z^{*}, n \in \N} G_{j+1/2}^{n} \frac{\varphi_{j}^{n}-\varphi_{j+1}^{n}}{\Delta x} + \sum_{n \in \N}\frac{\varphi_{0}^{n}}{\Delta x} G_{1/2,-}^{n} - \frac{\varphi_{1}^{n}}{\Delta x} G_{1/2,+}^{n} \\
&= \sum_{j \in \Z^{*}, n \in \N} G_{j+1/2}^{n} \frac{\varphi_{j}^{n}-\varphi_{j+1}^{n}}{\Delta x} + \sum_{n \in \N}\frac{\varphi_{0}^{n}}{\Delta x} (G_{1/2,-}^{n} -G_{1/2,+}^{n}) \\
& \qquad + \sum_{n \in \N} \frac{ \varphi_{0}^{n} -\varphi_{1}^{n}}{\Delta x} G_{1/2,+}^{n}.
\end{aligned}
$$
We almost have a discrete version of~\eqref{eq:DefFluid1}. The following lemma ensures that the corrective term $$\sum_{n \in \N}\frac{\varphi_{0}^{n}}{\Delta x} (G_{1/2,-}^{n} -G_{1/2,+}^{n})$$
 has the correct sign.
\begin{lemma} \label{lemma:GermDiss}
 If $g^{-}_{\lambda} - g^{+}_{\lambda}$ is nondecreasing with respect to its first two arguments then we have the dissipativity property
$$ G_{1/2,-}^{n}-G_{1/2,+}^{n} \geq 0.$$
\end{lemma}
\begin{proof}[Proof of Lemma~\ref{lemma:GermDiss}]
Let us denote by $a=u_{0}^{n} \top c_{0}$,  $\tilde{a}=u_{0}^{n} \bot c_{0}$, $b=u_{1}^{n} \top c_{1}$ and  $\tilde{b}=u_{1}^{n} \bot c_{1}$, such that $a \geq \tilde{a}$ and $b \geq \tilde{b}$. The dissipativity property holds if and only if
$$ g^{-}_{\lambda}(a,b,v^{n})-g^{-}_{\lambda}(\tilde{a},\tilde{b},v^{n})) \geq  g^{+}_{\lambda}(a,b,v^{n})-g^{+}_{\lambda}(\tilde{a},\tilde{b},v^{n}), $$
which is a straightforward consequence of the monotonicity of $g^{-}_{\lambda}- g^{+}_{\lambda}$ with respect to its two first variables.
\end{proof}
Let us go back to the proof of Lemma~\ref{prop:EI2}. Hypothesis~\eqref{dissipativity} exactly says that $g^{-}_{\lambda} - g^{+}_{\lambda}$ is nondecreasing with respect to its two first arguments. Thus we can apply Lemma~\ref{lemma:GermDiss} to obtain
$$ 
\sum_{j \in \Z, n \in \N} \frac{G_{j+1/2,-}^{n}-G_{j-1/2,+}^{n}}{\Delta x} \varphi_{j}^{n} \geq \sum_{j \in \Z^{*}, n \in \N} G_{j+1/2}^{n} \frac{\varphi_{j}^{n}-\varphi_{j+1}^{n}}{\Delta x} + \sum_{n \in \N} \frac{ \varphi_{0}^{n} -\varphi_{1}^{n}}{\Delta x} G_{1/2,+}^{n}  .
$$
Eventually, we have
$$ \sum_{j \in \Z, n \in \N} \eps_{j} \frac{A}{\Delta x} \dist_{1}((c_{-},c_{+}), \mathcal{H}_{\lambda}(v^{n})) \varphi_{j}^{n} = \frac{A}{\Delta x} \sum_{n \in \N}  \dist_{1}((c_{-},c_{+}), \mathcal{H}_{\lambda}(v^{n})) (\varphi_{0}^{n}+ \varphi_{1}^{n}) $$
and~\eqref{eq:EI2} is obtained by regrouping all the terms and changing their signs, and multiplying by $\Delta t \Delta x$.
\end{proof}

Passing to the limit $\Delta t \rightarrow 0$ in Equation~\eqref{eq:EI2}, we obtain the following proposition. 
\begin{prop}
If $u^{0}$ belongs to $BV(\R) \cap L^{1}(\R)$, if the CFL condition~\eqref{eq:CFL} holds and if Hypothesis~(\ref{gcons}-\ref{dissipativity}), (\emph{included~\eqref{eq:WBMax}}), are fulfilled, then the limit $u$ of $(u_{\Delta t})_{}$ verifies Inequality~\eqref{eq:DefFluid1} for any nonnegative function $\varphi$ in $\mathcal{C}_{0}^{\infty}(\R_{+} \times \R)$.
\end{prop}

\begin{proof} For small enough $\Delta t$, Condition~\eqref{eq:MassCond} is verified. Let us fix $(c_{-}, c_{+})$ in $\R^{2}$, and prove that for every nonnegative $\varphi$ in $\mathcal{C}_{0}^{\infty}$, the discrete inequality~\eqref{eq:EI2} converges to the continuous entropy inequality~\eqref{eq:DefFluid1}, where the sequence $(\varphi_{j}^{n})_{j \in \Z, n \in \N}$ is  defined by $\varphi_{j}^{n}= \varphi(n \Delta t, x_{j}^{n}-h^{n})$. We recall that $\mathcal{C}_{j}^{n}$ is the space-time cell 
 $$\mathcal{C}_{j}^{n}= \{(n \Delta t+s,x_{j-1/2}^{n}+y+s v^{n}), s \in [0, \Delta t), y \in [0, \Delta x ) \},$$
that $h^{n}$ is the discrete position of the particle's trajectory deduced from its velocity:
 $$ h^{n+1}= h^{n}+ v^{n} \Delta t $$
 and that the mesh is moving with the particle: $x_{j}^{n+1}= x_{j}^{n} + v^{n} \Delta t$. 
 We first treat the first term of~\eqref{eq:EI2}. The sequence of piecewise constant functions $(\zeta_{\Delta t})_{}$ defined by
 $$ \zeta_{\Delta t}(t,x)= \frac{\varphi_{j}^{n+1}-\varphi_{j}^{n}}{\Delta t} \ \ \text{ if } (t,x) \in \mathcal{C}_{j}^{n+1} $$ 
 converges uniformly to the function $(t,x) \mapsto (\partial_{t} \varphi) (t, x-h(t))$.
Indeed, for every $(t,x) \in \mathcal{C}_{j}^{n+1}$, there exists $\tilde{t} \in [n \Delta t, (n+1) \Delta t]$ such that
$$ 
\begin{aligned}
 |\zeta_{\Delta x}(t,x) -  (\partial_{t} \varphi) (t, x-h(t)) |	&= \left| \frac{\varphi((n+1) \Delta t, x_{j}^{n+1}- h^{n+1})-\varphi(n \Delta t, x_{j}^{n}- h^{n})}{\Delta t}  -  (\partial_{t} \varphi) (t, x-h(t)) \right| \\
 				&= |(\partial_{t } \varphi) (\tilde{t}, x_{j}^{n}- h^{n}) -  (\partial_{t} \varphi) (t, x-h(t)) | \\
				& \leq C (|\tilde{t}-t|+ |x- x_{j}^{n}| + |h^{n}-h(t)|) \\
				& \leq C (\Delta t+ \Delta x + ||h_{\Delta t}-h||_{\infty})
\end{aligned}
$$
We used the fact that $ x_{j}^{n+1}- h^{n+1}= x_{j}^{n}- h^{n}$. We conclude thanks to Remark~\ref{rem:convc}	 :
$$ 
\begin{aligned}
 \Delta t \Delta x \sum_{j \in \Z, n \in \N} |u_{j}^{n+1}-c_{j}| \frac{\varphi_{j}^{n+1}-\varphi_{j}^{n}}{\Delta t}  
 	&= \sum_{j \in \Z, n \in \N} \int_{\mathcal{C}_{j}^{n+1}} |u_{\Delta t}-c_{\Delta t}| \zeta_{\Delta t} dt \, dx \\
	&= \int_{\R} \int_{\R_{+}} \mathbf{1}_{t \geq \Delta t} |u_{\Delta t}-c_{\Delta t}| \zeta_{\Delta t} dt \, dx \\
	& \longrightarrow  \int_{\R} \int_{\R_{+}}  |u-c| (\partial_{t} \varphi) (t,x-h(t)) dt \, dx
\end{aligned}
$$
On the other hand,
$$ 
\begin{aligned}
 \Delta t \Delta x \sum_{j <0, n \in \N} G_{j+1/2}^{n} \frac{\varphi_{j+1}^{n}-\varphi_{j}^{n}}{\Delta x} = \int_{x< -\frac{\Delta x}{2}} \int_{\R_{+}} G_{\Delta t} \xi_{\Delta t} dt \, dx
\end{aligned}
$$
where for every $(t,x)$ in $\mathcal{C}_{j+1/2}^{n}= \{ (n \Delta t+s, x_{j}+y+v^{n}s), 0 \leq s < \Delta t, 0 \leq y < \Delta x  \}$, 
$$ 
\begin{aligned}
 G_{\Delta t}(t,x)= G_{j+1/2} = &g_{\lambda}\left(u_{\Delta t}\left (t, x-\frac{\Delta x}{2} \right) \top c_{-} ,u_{\Delta t} \left(t, x+\frac{\Delta x}{2} \right) \top c_{-}, v_{\Delta t}(t) \right) \\
 	&  - g_{\lambda}\left(u_{\Delta t}\left (t, x-\frac{\Delta x}{2} \right) \bot c_{-} ,u_{\Delta t} \left(t, x+\frac{\Delta x}{2} \right) \bot c_{-}, v_{\Delta t}(t) \right)
 \end{aligned} 
 $$
and for every $(t,x)$ in $\mathcal{C}_{j+1/2}^{n}$,
$$ \xi_{\Delta t} (t,x)= \frac{\varphi_{j+1}^{n}- \varphi_{j}^{n}}{\Delta x}. $$
The sequence $(\xi_{\Delta t})_{}$ converges uniformly to $(t,x) \mapsto \partial_{x} \varphi (t, x-h(t)) $. By continuity of translations in $L^{1}$, the sequences $(u_{\Delta t}(t, \cdot + \frac{\Delta x}{2}))_{\Delta t}$ and $(u_{\Delta t}(t, \cdot - \frac{\Delta x}{2}))_{\Delta t}$ converge in $L^{1}_{loc}$, and therefore up to extraction almost everywhere, toward $u$. On the other hand, $(v_{\Delta t})_{}$ converges almost everywhere toward $h'$. The consistency of the germ implies that $G_{\Delta t}$ converges almost everywhere to
$$ g(u \top c_{-}, u \top c_{-}, h') - g(u \bot c_{-}, u \bot c_{-}, h')= \sign(u-c_{-}) \left( \left( \frac{u^{2}}{2}- h' u \right) - \left( \frac{c_{-}^{2}}{2}- h' c_{-} \right) \right). $$
As $(u_{\Delta t})_{}$ and $(v_{\Delta t})_{}$ are uniformly bounded in $L^{\infty}$,  the dominated convergence theorem yields
$$  \Delta t \Delta x \sum_{j <0, n \in \N} G_{j+1/2}^{n} \frac{\varphi_{j+1}^{n}-\varphi_{j}^{n}}{\Delta x} \longrightarrow \int_{\R_{-}} \int_{\R_{+}} \Phi_{h'(t)}(u(t,x), c_{-}) \partial_{x} \varphi (t, x-h(t)) dt \, dx. $$
The second and fourth terms of~\eqref{eq:EI2} are easily treated:
$$ \Delta x \sum_{i \in \Z} |u_{j}^{0}-c_{j}| \varphi_{j}^{0} \longrightarrow \int_{\R} |u^{0}-c| \varphi(0,x) dx $$
and
$$ \Delta t \Delta x \sum_{ n \in \N} G_{j+1/2,+}^{n} \frac{\varphi_{1}^{n}-\varphi_{0}^{n}}{\Delta x} \longrightarrow 0. $$
Eventually, we study the convergence of
$$ \Delta t \sum_{n \in \N} \dist_{1}(c, \mathcal{H}_{\lambda}(v^{n})) (\varphi^{n}_{0}+\varphi^{n}_{1}) = 2 \int_{\R_{+}} \dist_{1}(c, \mathcal{H}_{\lambda}(v_{\Delta t})) \frac{\varphi_{\Delta t} (t, -\frac{\Delta x}{2})+\varphi_{\Delta t} (t, \frac{\Delta x}{2})}{2} dt.$$
Clearly, $\frac{\varphi_{\Delta x} (t, -\frac{\Delta x}{2})+\varphi_{\Delta x} (t, \frac{\Delta x}{2})}{2}$ converges uniformly to $\varphi( \cdot, 0)$. Moreover,
$$ 
\begin{aligned}
| \dist_{1}(c, \mathcal{H}_{\lambda}(v_{\Delta t})) - \dist_{1}(c, \mathcal{H}_{\lambda}(h')) | &= | \dist_{1}(c, (v_{\Delta t} -h',v_{\Delta t} -h' ) + \mathcal{H}_{\lambda}(h')) - \dist_{1}(c, \mathcal{H}_{\lambda}(h')) | \\
&= | \dist_{1}(c- (v_{\Delta t} -h',v_{\Delta t} -h' ) ,  \mathcal{H}_{\lambda}(h')) - \dist_{1}(c, \mathcal{H}_{\lambda}(h') )| \\
& \leq |v_{\Delta t}-h'|
\end{aligned}
$$
and
$$ \Delta t \sum_{n \in \N} \dist_{1}(c, \mathcal{H}_{\lambda}(v^{n})) (\varphi^{n}_{0}+\varphi^{n}_{1}) \longrightarrow 2 \int_{\R_{+}} \dist_{1}(c,\mathcal{H}_{\lambda}(h') ) \varphi(t,0) dt, $$
which concludes the proof.
\end{proof}

\begin{rem}
 In~\cite{CS12}, the authors are able to derive error estimates for the Godunov scheme adapted to a conservation law with a discontinuous flux (with respect to the space variable). The jump in such a flux can be related to the presence of the particle in our case, and a treatment partially consistent with the interface is also proposed in this paper. A careful investigation of the interface enables the authors to prove adapted $BV$ bounds, which are one of the main difficulties for obtaining error estimates. Due to the particular fluxes we use around the particle, we can also prove here $BV$ bounds, see Proposition~\ref{prop:FluidBound}, and one may expect to adapt the proof of~\cite{CS12} and thus obtain error estimates for our numerical methods.
\end{rem}

\subsection{Convergence of the particle's part}
We now prove that the limit $h$ of $(h_{\Delta t})$ verifies~\eqref{eq:DefPart}. To begin with, we prove that a discrete version of~\eqref{eq:DefPart} holds.
\begin{prop} \label{prop:DIP} 
 Let $(u_{j}^{n})_{n \in \N, j \in \Z }$ and $(v^{n})_{n \in \N}$ be given by Scheme~\eqref{eq:CoupledScheme}. Then, for every compactly supported sequences $(\xi^{n})_{n \in \N}$ and $(\psi_{j}^{n})_{n \in \N, j \in \Z}$ such that $\psi_{0}^{n}= \psi_{1}^{n}=1$ for all integrer $n$,
 \begin{equation} \label{eq:WFh} 
\begin{aligned}
-m \Delta t \sum_{n \in \N^{*}} v^{n} \frac{\xi^{n}- \xi^{n-1}}{\Delta t} = & \ m v^{0} \xi^{0} + \Delta x \Delta t \sum_{n \in \N^{*}, j \in \Z} u_{j}^{n}  \frac{\psi_{j}^{n} \xi^{n}- \psi_{j}^{n-1}\xi^{n-1}}{\Delta t}  \\
	&+ \Delta x \sum_{j \in \Z} u_{j}^{0} \xi^{0} \psi_{j} + \Delta t \Delta x \sum_{n \in \N, j \neq 0} f_{j+1/2}^{n} \xi^{n} \frac{\psi_{j+1}^{n}-\psi_{j}^{n}}{\Delta x}.
\end{aligned}
\end{equation}
\end{prop}
\begin{proof}
 We write
 $$ 
\begin{aligned}
 m \sum_{n \in \N} v^{n+1} \xi^{n} =  & \ m \sum_{n \in \N} v^{n} \xi^{n}+ \Delta t \sum_{n \in \N}(f_{1/2,-}^{n}- f_{1/2,+}^{n}) \xi^{n} \\
 	&+ \Delta x \sum_{n \in \N} \sum_{j \notin \{0, 1\}} \left[ (u_{j}^{n}-u_{j}^{n+1})- \mu (f_{j+1/2}^{n}-f_{j-1/2}^{n}) \right] \xi^{n} \psi_{j}^{n} \\
	&+ \Delta x \sum_{n \in \N}  \left[ (u_{0}^{n}-u_{0}^{n+1})- \mu (f_{1/2,-}^{n}-f_{-1/2}^{n}) \right] \xi^{n}  \\
	&+ \Delta x \sum_{n \in \N}  \left[ (u_{1}^{n}-u_{1}^{n+1})- \mu (f_{3/2}^{n}-f_{1/2,+}^{n}) \right] \xi^{n} .
\end{aligned}
$$
This comes from the fact that the sum of the last three lines is zero.
We now rearrange the different terms. On the one hand we have:
$$ \sum_{n \in \N,j \leq -1} (f_{j+1/2}^{n}-f_{j-1/2}^{n}) \xi^{n} \psi_{j}^{n} = \sum_{n \in \N,j \leq -1}  f_{j+1/2}^{n} \xi^{n} (\psi_{j}^{n}- \psi_{j+1}^{n}) + \sum_{n \in \N} \xi^{n} f_{-1/2}^{n},$$
and on the other hand we have:
$$ \sum_{n \in \N,j \geq 2} (f_{j+1/2}^{n}-f_{j-1/2}^{n}) \xi^{n} \psi_{j}^{n} = \sum_{n \in \N,j \geq 1}  f_{j+1/2}^{n} \xi^{n} (\psi_{j}^{n}- \psi_{j+1}^{n}) - \sum_{n \in \N} \xi^{n} f_{3/2}^{n}.$$
It follows that
 $$ 
\begin{aligned}
 m \sum_{n \in \N} v^{n+1} \xi^{n} =  & \ m \sum_{n \in \N} v^{n} \xi^{n}+ \Delta x \sum_{n \in \N, j\in \Z }(u_{j}^{n}-u_{j}^{n+1}) \xi^{n} \psi_{j}^{n} - \Delta t \sum_{n \in \N, j \neq 0} f_{j+1/2}^{n} \xi^{n} (\psi_{j}^{n}- \psi_{j+1}^{n}).
  \end{aligned}
$$
To conclude, we just have to rearrange the sum over $n$. Being careful with $n=0$  we obtain 
$$ \sum_{n \in \N} (v^{n+1}-v^{n}) \xi^{n} = \sum_{n \in \N^{*}} v^{n} (\xi^{n-1}- \xi^{n}) - v^{0} \xi^{0} $$ 
and 
$$ \sum_{n \in \N, j\in \Z }(u_{j}^{n}-u_{j}^{n+1}) \xi^{n} \psi_{j}^{n} =  \sum_{n \in \N^{*}, j\in \Z } u_{j}^{n} (\psi_{j}^{n} \xi^{n}- \psi_{j}^{n-1} \xi^{n-1}) + \sum_{j \in \Z} u_{j}^{0} \xi^{0} \psi_{j}^{0}, $$
and the result follows by regrouping all the terms.
\end{proof}

We can now pass to the limit $\Delta t \rightarrow 0$ in Proposition~\ref{prop:DIP} to prove that $h$ verifies~\eqref{eq:DefPart}.
\begin{prop}
Suppose that Hypothesis~(\ref{gcons}-\ref{dissipativity}) hold, and that the CFL condition~\eqref{eq:CFL} is fulfilled.
For all test functions $\xi$ and $\psi$ such that $\psi(0)=1$, the limit $h$ of $(h_{\Delta t})_{}$ verifies Inequality~\eqref{eq:DefPart}.
\end{prop}
 
\begin{proof}
Define
$$ \psi_{j}^{n}= \psi(x_{j}^{n}-h^{n}) \ \ \text{ and } \ \ \ \xi^{n}=\xi(n \Delta t). $$ 
Proposition~\ref{prop:DIP} applies if $\psi_{0}^{n}= \psi_{1}^{n}=1$. Here, we only have 
$$  \forall j \in \{ 0, 1 \}, \ \ \left| \psi_{j}^{n}-1\right| \leq C \Delta x . $$
The equality~\eqref{eq:WFh} holds up to the following corrections appearing in the left hand side:
$$ 
\begin{aligned}
\Delta x \Delta t \sum_{n \in \N^{*}, j \in \{0, 1\}} & u_{j}^{n} \frac{(1-\psi_{j}^{n}) \xi^{n} - (1-\psi_{j}^{n-1}) \xi^{n-1}}{\Delta t} + \Delta x \sum_{j \in \{0, 1\} } u_{j}^{0} \xi^{0} (1-\psi_{j}^{0}) \\
	& +\Delta x \Delta t \sum_{n \in \N} \left( f_{-1/2}^{n} \xi^{n} \frac{(1-\psi_{0}^{n})}{\Delta x} - f_{1/2}^{n} \xi^{n} \frac{(1-\psi_{1}^{n})}{\Delta x}  \right) ,
\end{aligned}
$$ 
which all tends to zero since $\psi_{0}^{n}- 1 = O(\Delta x)$ and  $\psi_{1}^{n}- 1 = O(\Delta x)$. The sequence
 $$ \zeta_{\Delta t}(t,x)= \frac{\psi_{j}^{n} \xi^{n}- \psi_{j}^{n-1}\xi^{n-1}}{\Delta t} \ \ \text{ if } (t,x) \in \mathcal{C}_{j}^{n}$$ 
 converges uniformly to the function $(t,x) \mapsto\psi \xi'$.
Indeed, by definition of the moving mesh, $x_{j}^{n}-h^{n}= x_{j}^{n-1}-h^{n-1}$. Therefore, $\psi_{j}^{n}= \psi_{j}^{n-1}$ and
$$ \frac{\psi_{j}^{n} \xi^{n}- \psi_{j}^{n-1}\xi^{n-1}}{\Delta t}  =\psi_{j}^{n} \frac{ \xi^{n}- \xi^{n-1}}{\Delta t} $$
which converges uniformly toward the expected function. Now, define $F_{\Delta t}$  by
$$ 
 F_{\Delta t}(t,x)= g_{\lambda}\left(u_{\Delta t}\left (t, x-\frac{\Delta x}{2} \right)  ,u_{\Delta t} \left(t, x+\frac{\Delta x}{2} \right) , v_{\Delta t}(t) \right) 
 $$
in such a way that for all $(t,x)$ in $\mathcal{C}_{j+1/2}^{n}$,
$$ F_{\Delta t} (t,x)= f_{j+1/2}^{n}. $$
By continuity of translations in $L^{1}$, the sequences $(u_{\Delta t}(t, \cdot + \frac{\Delta x}{2}))_{\Delta t}$ and $(u_{\Delta t}(t, \cdot - \frac{\Delta x}{2}))_{\Delta t}$ converge in $L^{1}_{loc}$, and therefore up to extraction, almost everywhere, toward $u$. On the other hand, $(v_{\Delta t})_{}$ converges almost everywhere toward $h'$. The consistency of the flux~\eqref{gcons} implies that $F_{\Delta t}$ converges almost everywhere to
$$ g(u , u , h')= \frac{u^{2}}{2} -h' u. $$

\end{proof}

\subsection{A family of scheme consistent with a maximal part of the germ} \label{sub:Example}
In this section we exhibit a family of schemes that verifies the set of Assumptions~(\ref{gcons}-\ref{dissipativity}). Let us clarify which maximal subset of $\mathcal{G}_{\lambda}$ is used.

\begin{prop} \label{prop:MaxSub} The part $\mathcal{H}_{\lambda}(v)= \mathcal{G}_{\lambda}^{1} \cup \mathcal{G}_{\lambda}^{2}(v)$ is a maximal subset of the germ.
\end{prop}
\begin{proof}
 Following~\cite{AS12} (see Equations (13) and (14) in this reference), it suffices to show that if 
\begin{equation}~\label{eq:maxG4}
\Xi_{v}((u_{-},u_{+}),(v_{-},v_{+})) \geq 0 \ \ \text{ for any } (v_{-},v_{+})\in \mathcal{G}_{\lambda}^{2}(v),
\end{equation}
 then the stronger following property holds
$$  \Xi_{v}((u_{-},u_{+}),(v_{-},v_{+})) \geq 0 \ \ \text{ for any } (v_{-},v_{+}) \in \mathcal{G}_{\lambda}^{1} \cup \mathcal{G}_{\lambda}^{2}(v). $$
In the sequel we suppose that $v=0$. The general case follows by translation. The two main arguments are first, that Proposition~\ref{def:MaxSub} implies that this is automatically verified if $(u_{-},u_{+})$ belongs to the germ, and second, that for all $(v_{-},v_{+})$ in $\mathcal{G}_{\lambda}^{2}$, $ \left| \frac{v_{-}^{2}- v_{+}^{2}}{2} \right| \leq \frac{\lambda^{2}}{2}$. In the sequel, $(v_{-},v_{+})$ always denotes an element of $\mathcal{G}_{\lambda}^{2}$. We proceed by a tedious, but not difficult, disjunction of cases. 

\begin{itemize}
\item If $u_{-} \geq \lambda$ and $u_{+} \geq 0$, then we want to prove that 
$$\frac{u_{-}^{2}-v_{-}^{2}}{2} -\frac{u_{+}^{2}-v_{+}^{2}}{2} \geq 0. $$
If we apply Equation~\eqref{eq:maxG4} to $(\lambda, 0)$, we obtain that
$$ \frac{u_{-}^{2}-u_{+}^{2}}{2} \geq \frac{\lambda^{2}}{2} $$
and the result follows.

\item If $0 \leq u_{-} \leq \lambda$ and $u_{+} \geq 0$, then $(u_{-},u_{+})$ belongs to the germ. Indeed, Equation~\eqref{eq:maxG4} applied to $(u_{-},0)$ yields $ -\frac{u_{+}^{2}}{2} \geq 0$ and therefore, $u_{+}=0$.

\item If $ u_{-} \leq 0$ and $u_{+} \geq 0$, then $(u_{-},u_{+})$ belongs to the germ. Indeed, Equation~\eqref{eq:maxG4} applied to $(0,0)$ yields 
$$ -\frac{u_{-}^{2}}{2} - \frac{u_{+}^{2}}{2}\geq 0$$ 
and therefore, $u_{-}=u_{+}=0$.

\item If $u_{-} \leq 0$ and $ -\lambda \leq u_{+} \leq 0$, then $(u_{-},u_{+})$ belongs to the germ. Indeed, Equation~\eqref{eq:maxG4} applied to $(0,u_{+})$ yields $ -\frac{u_{-}^{2}}{2} \geq 0$ and therefore, $u_{-}=0$.

\item If $u_{-} \leq 0$ and $  \leq u_{+} \leq -\lambda$, then we want to prove that 
$$-\frac{u_{-}^{2}-v_{-}^{2}}{2} +\frac{u_{+}^{2}-v_{+}^{2}}{2} \geq 0. $$
If we apply Equation~\eqref{eq:maxG4} to $(0, -\lambda)$, we obtain
$$-\frac{u_{-}^{2}}{2} +\frac{u_{+}^{2}-\lambda^{2}}{2} \geq 0. $$
and the result follows.

\item If $0 \leq u_{-} \leq \lambda$ and $ u_{+} \leq -\lambda$, let us first suppose that $u_{-} \geq v_{-}$. We have to prove that
$$ \frac{u_{-}^{2}-v_{-}^{2}}{2} + \frac{u_{+}^{2}-v_{+}^{2}}{2} \geq 0.$$
But $0 \leq v_{-} \leq u_{-}$ and $ 0 \geq v_{+} \geq u_{+}$, and we have the result:
$$ \frac{v_{-}^{2}+v_{+}^{2}}{2} \leq \frac{u_{-}^{2}+v_{+}^{2}}{2} \leq \frac{u_{-}^{2}+u_{+}^{2}}{2} . $$
We now suppose that $u_{-} \leq v_{-}$. We want to prove that
$$-  \frac{u_{-}^{2}-v_{-}^{2}}{2} + \frac{u_{+}^{2}-v_{+}^{2}}{2} \geq 0.$$
Moreover, $(u_{-}, u_{+})$ does not belong to the germ $\mathcal{G}_{\lambda}$, and therefore $u_{+} \leq -u_{-}- \lambda$ and
$$ \frac{u_{+}^{2}-u_{-}^{2}}{2} \geq \frac{2 u_{-} \lambda + \lambda^{2}}{2} \geq \frac{\lambda^{2}}{2} \geq \frac{v_{+}^{2}-v_{-}^{2}}{2}$$

\item If $\lambda \leq u_{-} $ and $ u_{+} \leq -\lambda$, the result
$$ \frac{u_{-}^{2}-v_{-}^{2}}{2} + \frac{u_{+}^{2}-v_{+}^{2}}{2} \geq 0 $$
is a straightforward consequence of
$$ \frac{u_{-}^{2}+u_{+}^{2}}{2} \geq \lambda ^{2} \geq  \frac{v_{-}^{2}+v_{+}^{2}}{2}.  $$

\item Eventually, if $\lambda \leq u_{-} $ and $ - \lambda \leq u_{+} \leq 0$, let us first suppose that $u_{+} \leq v_{+}$ and prove
$$ \frac{u_{-}^{2}-v_{-}^{2}}{2} + \frac{u_{+}^{2}-v_{+}^{2}}{2} \geq 0.$$
It follows from
$$  \frac{v_{+}^{2}+v_{-}^{2}}{2} \leq  \frac{u_{+}^{2}+v_{-}^{2}}{2} \leq  \frac{u_{+}^{2}+u_{-}^{2}}{2}.$$
Suppose now that $u_{+}>v_{+}$ and $ u_{+} \geq -u_{-}+ \lambda$. The result
$$ \frac{u_{-}^{2}-v_{-}^{2}}{2} - \frac{u_{+}^{2}-v_{+}^{2}}{2} \geq 0$$
comes from
$$ \frac{u_{-}^{2}-u_{+}^{2}}{2} \geq  \frac{-2 \lambda u_{+} + \lambda^{2}}{2} \geq  \frac{ \lambda^{2}}{2} \geq  \frac{v_{-}^{2}-v_{+}^{2}}{2}. $$
\end{itemize}

\end{proof}

It is possible to find fluxes that verifies~\eqref{eq:WBMax} with $ \mathcal{H}_{\lambda}=\mathcal{G}_{\lambda}^{1} \cup \mathcal{G}_{\lambda}^{2}$ and~\eqref{gpmmono}.
\begin{prop}
 The family of finite volume schemes defined by
 \begin{equation} \label{FVconsMax}
\begin{cases}
  g^{-}_{\lambda}(u_{-},u_{+},v)=  g(u_{-}, \min(u_{+}+\lambda, \max(u_{-},v)), v) \\
  g^{+}_{\lambda}(u_{-},u_{+},v)= g(\max(u_{-}-\lambda, \min(u_{+},v)), u_{+}, v) 
\end{cases}
\end{equation}
is consistent with $\mathcal{G}_{\lambda}^{1} \cup \mathcal{G}_{\lambda}^{2}(v)$ and verifies the monotonicity assumptions $\partial_{1} g^{\pm}_{\lambda} \geq 0$ and  $\partial_{2} g^{\pm}_{\lambda} \leq 0$.
\end{prop}

\begin{proof}
 The proof consists in a simple verification. We first check that for all $u_{-}$ and $u_{+}$ in $\R$,
 $$  g^{-}_{\lambda}(u_{-},u_{-}- \lambda,v)=  g(u_{-}, \min(u_{-}, \max(u_{-},v)), v) = g(u_{-},u_{-},v)$$
 and
 $$ g^{+}_{\lambda}(u_{+}+ \lambda,u_{+},v)= g(\max(u_{+}, \min(u_{+},v)), u_{+}, v) = g(u_{+},u_{+},v).$$
 Then, we verify that for all $u_{+}$ in $[v-\lambda, v]$, 
 $$ g^{-}_{\lambda}(v,u_{+},v)=  g(u_{-}, \min(u_{+}+\lambda, \max(v,v)), v) = g(v,v,v) $$
 and 
 $$   g^{+}_{\lambda}(v,u_{+},v)= g(\max(v-\lambda, \min(u_{+},v)), u_{+}, v) = g(u_{+},u_{+},v) $$
 while for every $u_{-}$ in $[v, v+\lambda]$,
 $$ g^{-}_{\lambda}(u_{-},v, v)=  g(u_{-}, \min(v+\lambda, \max(u_{-},v)), v)= g(u_{-},u_{-},v) $$
 and
 $$ g^{+}_{\lambda}(u_{-},v,v)= g(\max(u_{-}-\lambda, \min(v,v)), v, v)= g(v,v,v).  $$
Eventually, the monotonicity properties are implied by those on $g$ as soon as soon as the first component is not $u_{+}$ and the second is not $u_{-}$. But if the first component is $u_{+}$, then $u_{+} <v$ and $\partial_{2} g^{+}_{ \lambda}= u_{+}-v \leq 0$, while if the second component is $u_{-}$, then $u_{-} >v$ and $\partial_{2} g^{-}_{ \lambda}= u_{-}-v \geq 0$.
\end{proof}


It remains to prove that Assumption~\eqref{dissipativity} holds. This is not the case for every choice of flux $g$ (a counterexample can be found in~\cite{AS12}), but we can check it for three classical fluxes.
\begin{prop} \label{prop:SchemesThatWork}
 The family of finite volume schemes~\eqref{FVconsMax} verifies that $g^{-}_{\lambda}-g^{+}_{\lambda}$ is nondecreasing with respect to its two first variables if $g$ is the Godunov, the Rusanov or the Engquist--Osher numerical flux. 
\end{prop}
\begin{proof}
 Let us divide the phase space $(u_{-},u_{+})$ in six zones, depending on which values are taken by $g^{-}$ and $g^{+}$:
 $$ g^{-}_{\lambda}(u_{-},u_{+},v)= \left\{
\begin{array}{lll}
 g(u_{-},u_{-},v) & \text{ if } v \leq u_{-} \leq u_{+}+ \lambda 	&\text{ zone I,} \\
 g(u_{-},v, v) & \text{ if } u_{-} \leq v \leq u_{+}+ \lambda 	&\text{ zone II,} \\ 
 g(u_{-},u_{+}+\lambda,v) & \text{ if }  u_{+}+ \lambda \leq \max(u_{-},v)	&\text{ zone III,} 
\end{array}
\right.
$$
while
 $$ g^{+}_{\lambda}(u_{-},u_{+},v)= \left\{
\begin{array}{lll}
 g(u_{+},u_{+},v) & \text{ if } u_{-}- \lambda \leq u_{+} \leq v	&\text{ zone $1$,} \\
 g(v,u_{+}, v) & \text{ if } u_{-}- \lambda \leq v \leq u_{+} 		&\text{ zone $2$,} \\ 
 g(u_{-}-\lambda,u_{+},v) & \text{ if }  \min(u_{+},v) \leq u_{-}- \lambda	&\text{ zone $3$.} 
\end{array}
\right.
$$
\begin{psfrags}
 \psfrag{u_{-}}{$u_{-}$}
 \psfrag{u_{+}}{$u_{+}$}
 \psfrag{v}{$v$}
 \psfrag{v-l}{$v-\lambda$}
 \psfrag{v+l}{$v+\lambda$}
 \psfrag{I+1}{$I/1$}
 \psfrag{I+2}{$I/2$}
 \psfrag{I+3}{$I/3$}
 \psfrag{II+1}{$II/1$}
 \psfrag{III+1}{$III/1$}
 \psfrag{III+3}{$III/3$}
 \psfrag{II+2}{$II/2$}
 \psfrag{IV+2}{$IV+2$}
 \psfrag{II+4}{$II+4$}
 \psfrag{IV+4}{$IV+4$}
 \psfrag{IV+3}{$IV+3$}
 \psfrag{III+4}{$III+4$}
\begin{figure}[h!]
\centering
 \includegraphics[width=0.6\linewidth]{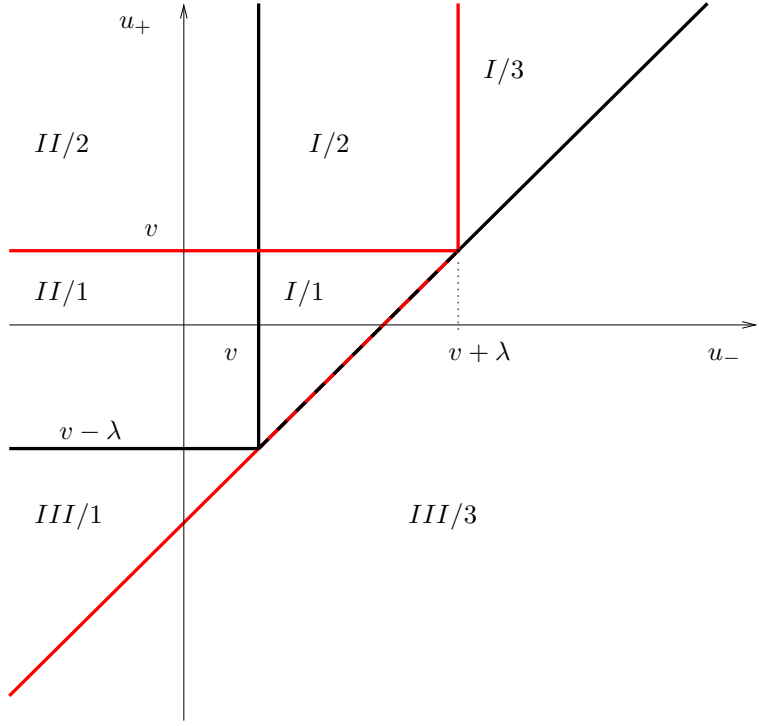}
 \caption{Choice of the fluxes in the family of finite volume schemes~\eqref{FVconsMax}.} \label{fig:zones}
\end{figure}
\end{psfrags}
These zones are depicted on Figure~\ref{fig:zones}. If $u_{+}$ belongs to zones $1$ or $2$, $g^{+}$ does not depends on $u_{-}$ and $g^{-}_{\lambda}-g^{+}_{\lambda}$ is nondecreasing with respect to its first argument. Similarly, if $u_{-}$ belongs to zones $I$ or $II$,  $g^{-}_{\lambda}-g^{+}_{\lambda}$ is nondecreasing towards its second argument. We focus on the case where $u_{-}$ belongs to zone $III$ or $u_{+}$ belongs to zone $3$. Let us first remark that the case where $u_{-}$ belongs to zone $III$ and $u_{+}$ is in zone $3$ reduces to the choice of flux studied in~\cite{AS12}, where the monotonicity property has been proven for the Godunov, Rusanov and Engquist--Osher scheme. Suppose that case $u_{-}$ is in zone $I$ and $u_{+}$ is in zone $3$.  Then we have
$$( g^{-}_{\lambda}-g^{+}_{\lambda})(u_{-},u_{+},v)= g(u_{-}, u_{-}, v) - g(u_{-}-\lambda, u_{+},v). $$
For the sake of simplicity we assume that $v=0$.
\begin{itemize}
 \item If $g$ is the Godunov flux, as $u_{+}+\lambda \geq u_{-} \geq \lambda$, the Riemann problem between $u_{-}-\lambda$ and $u_{+}$ is a shock traveling faster than $v$. It follows that $$( g^{-}_{\lambda}-g^{+}_{\lambda})(u_{-},u_{+},0)= \frac{(u_{-})^{2}}{2}- \frac{(u_{-}- \lambda)^{2}}{2}= \lambda u_{-}- \frac{\lambda^{2}}{2}$$
is nondecreasing toward its first two arguments. 
\item If $g$ is the Rusanov flux,
$$( g^{-}_{\lambda}-g^{+}_{\lambda})(u_{-},u_{+},0)= \frac{(u_{-})^{2}}{2}-\left( \frac{(u_{-}- \lambda)^{2}+ u_{+}^{2}}{4} - (u_{-}- \lambda) \frac{u_{+}-(u_{-}-\lambda)}{2} \right)$$
and we have 
$$ 
\begin{aligned}
 \partial_{1}( g^{-}_{\lambda}-g^{+}_{\lambda})(u_{-},u_{+},0) &= u_{-}- \left( \frac{u_{-}- \lambda}{2} -  \frac{u_{+}-(u_{-}-\lambda)}{2} + \frac{u_{-}-\lambda}{2}\right) \\
 	&= \frac{-u_{-}+3 \lambda+u_{+}}{2}. 
\end{aligned}
$$
As $u_{-}$ belongs to zone $I$,  $u_{+}+ \lambda \geq u_{-}$, and the last quantity is larger than $\lambda$. On the other hand,
$$  \partial_{2}( g^{-}_{\lambda}-g^{+}_{\lambda})(u_{-},u_{+},0) = - \frac{u_{+}-(u_{-}- \lambda)}{2}$$
and this last quantity is nonnegative because $u_{+}$ belongs to zone $3$. 

\item Eventually, if $g$ is the Engquist--Osher scheme, the fact that $0 \leq u_{-}-\lambda \leq u_{+}$ implies that 
$$ ( g^{-}_{\lambda}-g^{+}_{\lambda})(u_{-},u_{+},0)= \frac{(u_{-})^{2}}{2} - \frac{(u_{-}- \lambda)^{2}}{2} = \lambda u_{-} - \frac{\lambda^{2}}{2}$$
is once again nondecreasing with respect to its first two arguments. The case where $u_{-}$ is in zone $III$ while $u_{+}$ is in zone $1$ can be treated in a symmetrical way.

\end{itemize}

%

\end{proof}

\section{Convergence of schemes only consistent with $\mathcal{G}_{\lambda}^{1}$} \label{S:ConvLine}
In this section, we no longer require Hypothesis~\eqref{eq:WBMax} to be fulfilled, and prove convergence of a family of finite volume schemes that verifies only~\eqref{eq:WBG1}. The difficulty is that $\mathcal{G}_{\lambda}^{1}$ is not a maximal part of the germ, and we cannot prove a discrete version of~\eqref{eq:DefFluid1} directly. The key point is to study the convergence of the solution of Scheme~\eqref{eq:CoupledScheme} for initial data in the maximal subset of the germ $\mathcal{G}_{\lambda}^{1} \cup \mathcal{G}_{\lambda}^{2}$. We then extend the comparison argument of~\cite{AS12} to prove convergence for arbitrary initial data.
\subsection{Proof of convergence} \label{SG1}
Let us now focus on fluxes that do not preserve a maximal part of the germ (in the sense of Hypothesis~\eqref{eq:WBMax}), but only the straight line $\mathcal{G}_{\lambda}^{1}$, i.e. that verifies~\eqref{eq:WBG1} but not~\eqref{eq:WBMax}. Our aim is to prove the following theorem.
\begin{thm} \label{thm:convG1}
 If the numerical fluxes around the particle are given by
 \begin{equation} \label{eq:DragSchemeLambda}
\begin{cases}
f_{1/2,-}^{n}(u_{0}^{n},u_{1}^{n}, v^{n})= g(u_{0}^{n}, u_{1}^{n}+ \lambda, v^{n}), \\
f_{1/2,+}^{n}(u_{0}^{n},u_{1}^{n}, v^{n})= g(u_{0}^{n}-\lambda, u_{1}^{n}, v^{n}),
\end{cases}
\end{equation}
where $g$ is a numerical flux verifying~(\ref{gcons}-\ref{eq:WBG1}) and ~(\ref{eq:WBTip}-\ref{dissipativity}), and if the CFL condition~\eqref{eq:CFL} holds, Scheme~\eqref{eq:CoupledScheme} converges toward the solution of~\eqref{eq:CauchyPb}.
\end{thm}
\begin{proof}
Let us first remark that Proposition~\ref{prop:FluidBound}  and Proposition~\ref{prop:PartBound} did not use Hypothesis~\eqref{eq:WBMax}, thus we can extract converging subsequences as we did in the previous Section. Now, consider a test function $\varphi$ supported in $\{x<0\}$ or $\{x>0\}$, we have $\varphi_{0}^{n}=\varphi_{1}^{n}=0$ for small enough $\Delta x$. We easily obtain, as in Proposition~\ref{prop:EntIneq2Prel}, that for all $c$ in $\R$, for all $j \leq -1$,
$$ \frac{|u_{j}^{n-1}-c|-|u_{j}^{n}-c|}{\Delta t} + \frac{G_{j+1/2}^{n}- G_{j-1/2}^{n}}{\Delta x} \leq 0. $$ 
Multiplying by $\Delta t \Delta x \varphi_{j}^{n}$ and summing over $n \in \N$ and $j \leq -1$, we obtain as in Proposition~\ref{prop:EI2}
 \begin{equation*}
 \Delta t \Delta x \sum_{j \in \Z, n \leq  -1} |u_{j}^{n+1}-c| \frac{\varphi_{j}^{n+1}-\varphi_{j}^{n}}{\Delta t} + \Delta x \sum_{i \in \Z} |u_{j}^{0}-c| \varphi_{j}^{0} + \Delta t \Delta x \sum_{j \in \Z^{*}, n \leq -1} G_{j+1/2}^{n} \frac{\varphi_{j+1}^{n}-\varphi_{j}^{n}}{\Delta x} \geq 0
\end{equation*}
and we straightforwardly obtain that the limit $u$ of the scheme is an entropy solution of the Burgers equation on the sets $\{x<h\}$ (and similarly on $\{x>h\}$). 
 It remains to prove that the traces around the particle belong to the germ for almost every time. Let us fix a time $t_{0}$ such that $h'$ and the traces $u_{-}(t_{0})$ and $u_{+}(t_{0})$ exist. Fix $(c_{-},c_{+})$ in $\mathcal{H}_{\lambda}(h'(t_{0}))$. Our aim is to prove a discrete version of~\eqref{eq:DefFluid1}.
 Let us first suppose that $(c_{-}, c_{+})$ belongs to the straight line $\mathcal{G}_{\lambda}^{1}$ but not to the closed square $\overline{\mathcal{G}_{\lambda}^{2}(h'(t_{0}))}$. By continuity of $h'$, there exists $\delta>0$ such that,
$$\forall t \in (t_{0}-\delta, t_{0}+ \delta), \ \dist_{1}((c_{-},c_{+}), \mathcal{G}_{\lambda}^{1}) =  \dist_{1}((c_{-},c_{+}), \mathcal{H}_{\lambda}^{1}(h'(t)))$$
(see Figure~\ref{F:zones2}).
Up to taking a smaller $\delta$, this equality is also true at the numerical level for small enough $\Delta t$, since from Proposition~\ref{thm:Extraction}, $(v_n)_{n\in\N}$ converges. Therefore, passing to the limit in~\eqref{eq:EI2} with $\varphi$ supported in time in $(t_{0}-\delta, t_{0}+\delta)$, we directly obtain~\eqref{eq:DefFluid1}.

We now treat the case where $(c_{-},c_{+})$ belongs to the interior of $\mathcal{G}_{\lambda}^{2}(h'(t_{0}))$. The principle of the proof is to compare the numerical solution with another one, for which the initial data is much simpler as it corresponds to an element of $\mathcal{G}_{\lambda}^{2}(h'(t_{0}))$. Since $h'$ is continuous, there exists $\delta $ such that
$$ \forall t \in (t_{0}-\delta, t_{0}+ \delta), \ (c_{-},c_{+}) \in \mathcal{G}_{\lambda}^{2}(h'(t)) $$
and on the time interval $(t_{0}-\delta, t_{0}+ \delta)$,~\eqref{eq:DefFluid1} becomes
\begin{equation} \label{eq:DefFluid2}
  \dis \int_{\R_{+}} \int_{\R}  |u-c|(s,x) \partial_{t} \varphi (s,x-h(s))+  \Phi_{h'(t)}(u,c)(s,x) \partial_{x} \varphi(s,x-h(s)) dx\, ds \geq 0.
\end{equation}
Up to reducing $\delta$ and for small enough $\Delta t$, this is also true at the numerical level.
Now, for $(u_{j}^{n})_{j \in \Z, n \in \N}$ and $(v^{n})_{n \in \N}$ given by the fully coupled scheme~\eqref{eq:CoupledScheme}, consider $(c_{j}^{n})_{j \in \Z, n \in \N^{*}}$  the sequence given by the scheme
\begin{equation} \label{eq:OneWayScheme}
 \begin{cases}
 c_{j}^{n+1} &= c_{j}^{n}- \mu (g(c_{j}^{n}, c_{j+1}^{n},v^{n}) - g(c_{j-1}^{n}, c_{j}^{n},v^{n}))  \text{ for $j \notin \{0, 1\}$}, \\
 c_{0}^{n+1} &= c_{0}^{n}- \mu (g(c_{0}^{n}, c_{1}^{n}+ \lambda,v^{n}) - g(c_{-1}^{n}, c_{0}^{n},v^{n}) ), \\
 c_{1}^{n+1} &= c_{1}^{n}- \mu (g(c_{1}^{n}, c_{2}^{n},v^{n}) - g(c_{0}^{n}- \lambda, c_{1}^{n},v^{n})),
\end{cases}
\end{equation}
with initial data
\begin{equation} \label{eq:OWSInitial}
  c_{j}^{0}= 
\begin{cases}
 c_{-} & \text{ if } j \leq 0, \\
 c_{+} & \text{ if } j > 0.
\end{cases}
\end{equation}
We recall that $(c_{-},c_{+})$  belongs to $\mathcal{G}_{\lambda}^{2}(h'(t_{0}))$.
Simple modifications of Propositions~\ref{prop:EntIneq2Prel} and~\ref{prop:EI2} yield
$$
 \begin{aligned}
 \Delta t &\Delta x \sum_{j \in \Z, n \in \N} |u_{j}^{n+1}-c_{j}^{n+1}| \frac{\varphi_{j}^{n+1}-\varphi_{j}^{n}}{\Delta t} + \Delta x \sum_{i \in \Z} |u_{j}^{0}-c_{j}^{0}| \varphi_{j}^{0} \\
 & + \Delta t \Delta x \sum_{j \in \Z^{*}, n \in \N} G_{j+1/2}^{n} \frac{\varphi_{j+1}^{n}-\varphi_{j}^{n}}{\Delta x}  + \Delta t \Delta x \sum_{ n \in \N} G_{j+1/2,+}^{n} \frac{\varphi_{1}^{n}-\varphi_{0}^{n}}{\Delta x} \geq 0.
\end{aligned}
$$
Suppose that $(c_{j}^{n})_{j \in \Z, n \in \N}$ converges to $c(t,x) = c_{-} \mathbf{1}_{x<h(t)} + c_{+} \mathbf{1}_{x>h(t)} $ on the interval $( t_{0}-\delta, t_{0}+ \delta )$. Then with $\varphi_{j}^{n}= \varphi(t^{n}, x_{j}^{n})$ where $\varphi$ is a test function supported in $(t_{0}-\delta, t_{0}+ \delta)$, we obtain~\eqref{eq:DefFluid2} by passing to the limit. We now study this convergence.

\begin{lemma} \label{lemma:PropScheme}
Suppose that at iteration $n$, the sequence $(c_{j}^{n})_{j \in \Z}$ given by the scheme~\eqref{eq:OneWayScheme} is nondecreasing on $j \leq 0$ and on $j \geq 1$, and such that 
$$ \forall j \leq  0, \, c_{-} \leq c_{j}^{n} \leq c_{-}+ \lambda \ \ \text{ and } \ \   \forall j \geq 1, \,  c_{+}- \lambda \leq c_{j}^{n} \leq c_{+} $$
and
$$ c_{0}^{n}-c_{1}^{n} \leq \lambda, $$
then the same holds at iteration $n+1$.
\end{lemma}
\begin{proof}
 The monotonicity of $(c_{j}^{n+1})_{j \leq 0}$ follows from the monotonicity of $H_{\lambda}$ under the CFL condition~\eqref{eq:CFL}. For $j \leq -2$, we have
 $$ c_{j}^{n+1} = H_{\lambda}(c_{j-1}^{n}, c_{j}^{n}, c_{j+1}^{n}) \leq H_{\lambda}(c_{j}^{n}, c_{j+1}^{n}, c_{j+2}^{n}) = c_{j+1}^{n+1}. $$
As $c_{0}^{n} \leq c_{1}^{n}+ \lambda$, we also have
  $$ c_{-1}^{n+1} = H_{\lambda}(c_{-2}^{n}, c_{-1}^{n}, c_{0}^{n}) \leq H_{\lambda}(c_{-1}^{n}, c_{0}^{n}, c_{1}^{n}+\lambda) = c_{0}^{n+1}. $$
 Moreover, for $j \leq- 1$, both $c_{j-1}^{n}$, $c_{j}^{n}$ and $c_{j+1}^{n}$ are between $c_{-}$ and $c_{-}+ \lambda$, thus the same holds at iteration $n+1$. For $j=0$, as $c_{+} \leq c_{-}$ (because $(c_{-},c_{+})$ belongs to $\mathcal{G}_{\lambda}^{2}(h'(t_{0}))$), we conclude by remarking that
 $$c_{-} \leq c_{0}^{n} \leq c_{1}^{n} + \lambda \leq c_{+} + \lambda \leq c_{-} + \lambda.$$
 The results for positive integers $j$ are obtained in a similar way. Let us now prove that $u_{0}^{n+1}- u_{1}^{n+1} \leq \lambda$. We have
 \begin{align*}
 c_{0}^{n+1}- c_{1}^{n+1} &= H_{\lambda}(c_{-1}^{n}, c_{0}^{n}, c_{1}^{n}+ \lambda) - H_{\lambda}(c_{0}^{n} - \lambda, c_{1}^{n}, c_{2}^{n}) \\
 	& \leq H_{\lambda}(c_{0}^{n}, c_{0}^{n}, c_{1}^{n}+ \lambda) - H_{\lambda}(c_{0}^{n}- \lambda, c_{1}^{n}, c_{1}^{n}) \\
	& \leq c_{0}^{n} + \mu L | c_{0}^{n}- (c_{1}^{n}+ \lambda) | -c_{1}^{n }+ \mu L | (c_{0}^{n}- \lambda)- c_{1}^{n} | \\
	& \leq c_{0}^{n}- c_{1}^{n} + (c_{1}^{n} + \lambda - c_{0}^{n}) \\
	& \leq \lambda.
\end{align*}

\end{proof}
For $(c_{-},c_{+})$ in the \emph{open} subset $\mathcal{G}_{\lambda}^{2}(h'(t_{0}))$, there exists a positive $\delta$ such that $h'(t)$ stays in the interval $(c_{+}, c_{-})$ on the time interval $(t_{0}- \delta, t_{0}+ \delta)$. For small enough $\Delta t$, it is also true at the numerical level. 
Up to reducing slightly $\delta$, $(c_{-},c_{+})$ belongs to $\mathcal{G}_{\lambda}^{2}(v^{n})$ for small enough $\Delta t$ and for all iteration in time such that $t^{n}$ belongs to $(t_{0}-\delta, t_{0}+ \delta)$, and in particular $c_{+} \geq v^{n} \geq c_{-}$.

Thus the limit $c$ of the scheme~\eqref{eq:OneWayScheme} with initial data~\eqref{eq:OWSInitial} at time $t_{0}- \delta$ is such that $c$ is larger than $h'$ on $x<h$ and smaller on $x>h$. It allows to prove that $c$ is, on $\{ (t,x): x<h(t)\}$, the solution of 
\begin{equation} \label{eq:BLN}
 \begin{cases}
 \partial_{t} u + \partial_{x} \frac{u^{2}}{2} =0 &\forall t \in (t_{0}- \delta, t_{0}+ \delta), \forall x<h(t) , \\
 u(t_{0}- \delta ,x)= c_{-} &\forall x<h(0) \\
 u(t,h(t))=h'(t) &\forall t \in (t_{0}- \delta, t_{0}+ \delta).
\end{cases}
\end{equation}
As $c_{-}$ is larger than $h'$ on the whole time interval, the boundary condition is inactive and the solution is $u=c_{-}$. Let us recall the definition given by Bardos, LeRoux and Nedelec in~\cite{BLN79} of this conservation law on a bounded domain. A function $u$ in $L^{\infty}$ is a solution of
$$ 
\begin{cases}
 \partial_{t} u + \partial_{x} f(u) = 0 & \forall t>0, \forall x<h(t), \\
 u(t=0,x)= u^{0}(x) &  \forall x<h(0), \\
 u(t, h(t))= u_{b}(t) & \forall t>0, 
\end{cases}
$$
 if for all real $\kappa$ and for all nonnegative function $\varphi \in \mathcal{C}_{0}^{\infty}(\R_{+} \times \R)$, the following inequality holds:
\begin{equation} \label{eq:RPBoundary}
\begin{aligned} 
  \int_{t>0} & \int_{x<h(t)}  |u(t,x)-\kappa| \partial_{t} \varphi(t,x-h(t)) + \Phi_{h'(t)}(u(t,x),\kappa) \partial_{x} \varphi(t,x-h(t)) dx \, dt  \\ &+ \int_{x<h(0)} |u^{0}(x)- \kappa| \varphi(0,x) dx 
   + \int_{t>0} \sign( \kappa-u_{b}(t))\{ f(u(t,h(t)^{-}))-f(\kappa)\} \varphi(t,0) \geq 0.
\end{aligned}
\end{equation}
The convergence of finite volume schemes for scalar conservation laws in a bounded domain has been proven in~\cite{V02} for instance. We are here in a favorable case: we can obtain a discrete version of~\eqref{eq:RPBoundary} by summing ~\eqref{eq:DIE} multiplied by $\Delta t \,  \Delta x \,  \varphi_{j}^{n}$ over $n \geq 0$ and $j \leq -1$. We obtain
\begin{align*}
\Delta t \Delta x \sum_{n \geq 0, j \leq -1} |c_{j}^{n+1}- \kappa| \frac{\varphi_{j}^{n+1}-\varphi_{j}^{n}}{\Delta t} + \Delta x \sum_{j \leq -1} |c_{j}^{0}- \kappa| \varphi_{j}^{0} \\
+ \Delta t \Delta x \sum_{n \geq 0, j \leq -1}  G_{j+1/2}^{n} \frac{\varphi_{j+1}^{n}- \varphi_{j}^{n}}{\Delta x} - \Delta t \sum_{n \geq 0} G_{-1/2}^{n} \varphi_{0}^{n} \leq 0.
\end{align*}
Passing to the limit yields 
\begin{align*} 
  \int_{t>0} \int_{x<h(t)} & |c(t,x)-\kappa| \partial_{t} \varphi(t,x-h(t)) + \Phi_{h'(t)}(c(t,x),\kappa) \partial_{x} \varphi(t,x-h(t)) dx \, dt  \\ &+ \int_{x<h(0)} |c_{-}- \kappa| \varphi(0,x) dx 
   + \int_{t>0} \sign( \kappa-c(t,h(t)))\{ f(c(t,h(t)^{-}))-f(\kappa)\} \varphi(t,0) \geq 0.
\end{align*}
To conclude we check that
$$ \sign( \kappa-h'(t))\{ f(c(t,h(t)^{-}))-f(\kappa)\}  \geq  -\sign( c(t,h(t)-\kappa))\{ f(c(t,h(t)^{-}))-f(\kappa)\}. $$
This relies strongly on the fact that $c$ remains larger than $h'$.
\begin{itemize}
 \item If $h' \leq \kappa \leq c$, the inequality reduces to
 $$\{ f(c(t,h(t)^{-}))-f(\kappa)\}  \geq  -\{ f(c(t,h(t)^{-}))-f(\kappa)\}$$
 which holds because $f$ is increasing on $(0, + \infty)$.
 \item If $h' \leq c \leq \kappa$ or $ \kappa \leq h' \leq c $ the inequality reduces to
  $$\{ f(c(t,h(t)^{-}))-f(\kappa)\}  \geq  \{ f(c(t,h(t)^{-}))-f(\kappa)\}$$
 or
  $$- \{ f(c(t,h(t)^{-}))-f(\kappa)\}  \geq  - \{ f(c(t,h(t)^{-}))-f(\kappa)\},$$
  which are both trivial.
\end{itemize}
\end{proof}

\begin{rem}
Of course, Theorem~\ref{thm:convG1} applies when the initial data is
$$ u^{0}(x)= c_{-} \mathbf{1}_{x<0} + c_{+} \mathbf{1}_{x \geq 0}, $$
with $(c_{-},c_{+}) \in \mathcal{G}_{\lambda}^{2}(v^{0})$. In Appendix~\ref{Append}, we prove the convergence for this specific initial data directly, without using the local in time comparison with the one-way scheme~\eqref{eq:OneWayScheme} in which the velocity of the particle is fixed.

\end{rem}


\appendix
\appendixpage

\section{Detailed analysis when the initial data belongs to $\mathcal{G}_{\lambda}^{2}(v^{0})$} \label{Append}

Our aim in this section is to prove directly that if
\begin{equation} \label{eq:RID}
 u^{0}(x)= u_{-} \mathbf{1}_{x<0} + u_{+} \mathbf{1}_{x \geq 0}  \quad \text{and} \quad h^{0}=0,
\end{equation}
with $(u_{-},u_{+})$ in $\mathcal{G}_{\lambda}^{2}(v^{0})$,  Scheme~\eqref{eq:CoupledScheme} converges toward the exact solution, which in that case is given by
$$ 
\begin{cases}
 h(t)=\frac{u_{-}+u_{+}}{2}t + \left( v^{0} - \frac{u_{-}+u_{+}}{2}\right) \frac{m_{p}}{u_{-}-u_{+}} \left(1 - e^{- \frac{u_{-}-u_{+}}{m_{p}}t} \right), \\
 u(t,x)= u_{-} \mathbf{1}_{x<h(t)} + u_{+} \mathbf{1}_{x \geq h(t)}.
\end{cases}
 $$

In this section only and for technical reasons, we consider a finite volume scheme on a bounded space domain $[-a,a]$, subdivided with $2 M_{c}$ cells and with periodic boundary conditions. The scheme under consideration writes
 \begin{equation} \label{eq:CoupledSchemeP}
 \begin{cases}
 u_{j}^{n+1} &= u_{j}^{n}- \mu (g(u_{j}^{n}, u_{j+1}^{n}, v^{n}) - f(u_{j-1}^{n}, u_{j}^{n}, v^{n}))  \text{ for $j \in \{ -M_{c}+1, \cdots, M_{c} \} \setminus \{0, 1\}$}, \\
 u_{0}^{n+1} &= u_{0}^{n}- \mu (g(u_{0}^{n}, u_{1}^{n}+ \lambda, v^{n})- f(u_{-1}^{n}, u_{0}^{n},v^{n})), \\
 u_{1}^{n+1} &= u_{1}^{n}- \mu (g(u_{1}^{n}, u_{2}^{n},v^{n}) - g(u_{0}^{n}-\lambda, u_{1}^{n}, v^{n})), \\
 u_{-M_{c}}^{n}& = u_{M_{c}}^{n} \quad \text{ and } \quad u_{M_{c}+1}^{n}= u_{-M_{c}+1}^{n}, \\
 v^{n+1} & = v^{n} + \frac{\Delta t}{m_{p}} (g(u_{0}^{n}, u_{1}^{n}+\lambda, v^{n})-g(u_{0}^{n}-\lambda, u_{1}^{n}, v^{n}), \\
 x_{j}^{n+1}& =x_{j}^{n}+ v^{n}\Delta t.
\end{cases}
\end{equation}
We recall that the ratio of the time step $\Delta t$ and the cell size $\Delta x$ is equals to $\mu$. We fixed the final time~$T$. At each time step, four new cells (one of both part of the particle and one of each extremities of the interval because of the periodic boundary conditions) are influenced by Scheme~\eqref{eq:CoupledSchemeP}, in the sense that their values were constant equals to $u_{-}$ or $u_{+}$ before. We take $a$ large enough so that the influence of the particle does not interact with the influence of the boundary condition, and stays in the interval $[-a/3,a/3]$ during the time interval $[0,T]$ (see Figure~\ref{F:Perio} below). This is achieved by taking $a$ larger than $\frac{3T}{\mu}$.
\begin{psfrags}
\psfrag{a}{$\frac{M_{c}}{3}$ cells influenced}
\psfrag{b}{by the boundary}
\psfrag{c}{$\frac{M_{c}}{3}$ cells influenced by the}
\psfrag{d}{particle on each side}
\psfrag{-a}{$-a$}
\psfrag{aa}{$a$}
\psfrag{u-}{$u_{-}$}
\psfrag{u+}{$u_{+}$}
 \begin{figure}[htp]
\centering
 \includegraphics[width=0.9\linewidth]{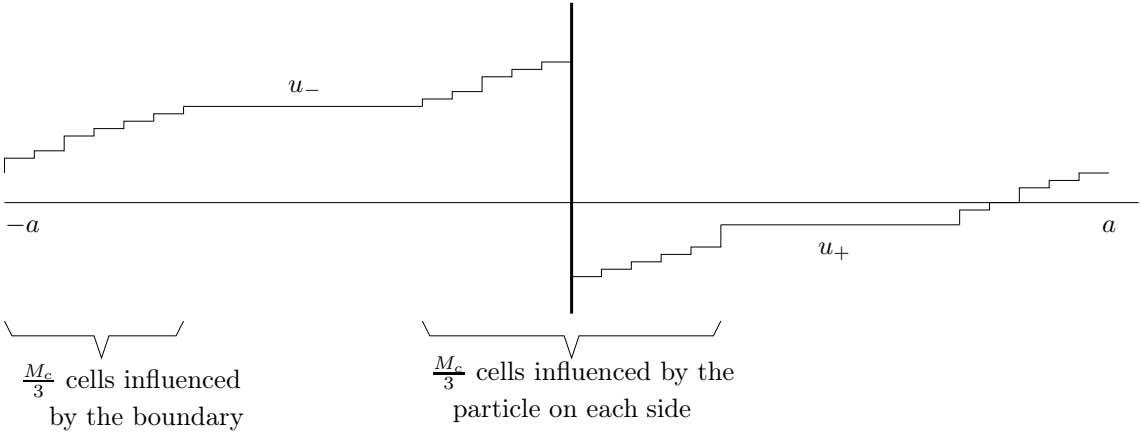}
 \caption{Shape of the numerical solution at time $T$. If $a$ is large enough, the contribution of the particle and of the boundary conditions remain separated.} \label{F:Perio}
\end{figure}
\end{psfrags}
The next proposition states that Scheme~\eqref{eq:CoupledSchemeP} converges toward the solution of the fully coupled problem~\eqref{eq:CauchyPb}.

\begin{prop}
 Suppose that the numerical flux $g$ verifies~(\ref{gcons}-\ref{eq:WBG1}) and ~(\ref{eq:WBTip}-\ref{dissipativity}), that
 \begin{equation} \label{eq:FluxSym}
 \forall A \in \R, \forall B \in \R, \quad g(v-A, v-B,v)= g(v+B,v+A,v),
\end{equation}
and that $\partial_{3} g$ is decreasing with respect to its first two arguments. Under Condition~\eqref{eq:CFL} and for the initial data~\eqref{eq:RID}, Scheme~\eqref{eq:CoupledSchemeP} converges toward the solution of~\eqref{eq:CauchyPb} on
$$ \{ (t,x): t<T \quad \text{and} \quad -a/3+h(t)<x<a/3+h(t)\}. $$
\end{prop}
\begin{proof}
 We prove, as we did in Section~\ref{S:ConvLine}, that $(u_{j}^{n})_{ -M_{c}/3 \leq j \leq 0}$ converges toward the solution of~\eqref{eq:BLN}, with a Neumann boundary condition on the left of the particle. The key point is to prove that $v^{n}$ remains smaller than $c_{-}$ on the whole time interval $[0,T]$, in which case the boundary condition is inactive and we obtain the result. Similarly on the right of the particle, the boundary condition is inactive if $v^{n}$ remains larger than~$c_{+}$.
 
 To prove that $c_{+} \leq v^{n} \leq c_{-}$, we apply the Crandall--Tartar lemma~\cite{CT80} to the application
  $$ 
 \begin{array}{lclc}
 T:	& \mathcal{S} & \longrightarrow & \mathcal{S} \\
 	&((u_{j}^{0})_{-M_{c}+1 \leq j \leq M_{c}}, v^{0}) & \longmapsto & ((u_{j}^{n})_{-M_{c}+1 \leq j \leq M_{c}}, v^{n}).
\end{array}
 $$
 where 
 $$ \mathcal{S}=\{((b_{j})_{j \in \{ -M_{c}+1, \cdots, M_{c} \}},v): \ b_{1} \leq b_{2} \leq \cdots  \leq b_{M_{c}} \leq b_{-M_{c}+1} \leq b_{-M_{c}+2} \leq \cdots \leq  b_{-1} \leq b_{0} \leq b_{1}+ \lambda \}. $$
 
\begin{lemma}[Crandall--Tartar]
 Let $(\Omega, \mu)$ be a measured space, and let $\mathcal{S}$ be a subset of $L^{1}(\Omega)$ stable by sup:
 $$ \forall (u,v) \in \mathcal{S}^{2}, \quad \max(u,v) \in \mathcal{S}. $$
 Consider a function $T: \mathcal{S} \rightarrow \mathcal{S}$  such that
 $$ \forall u \in \mathcal{S}, ||T(u) ||_{L^{1}} = || u ||_{L^{1}}$$
Then, if $T$ is order preserving,
 $$  ||T(u)-T(v)||_{L^{1}} \leq ||u-v||_{L^{1}}  $$
\end{lemma}
 In our case, $\Omega= \R^{2M_{c}} \times \R$ and
 $$ ||(b_{j})_{j \in \{ -M_{c}+1, \cdots, M_{c} \}}, v ||_{L^{1}} = \Delta x \sum_{j=-M_{c}+1}^{M_{c}} |b_{j}| + m |v|. $$
It is straightforward to verify that Scheme~\eqref{eq:CoupledSchemeP} preserves the norm $||\cdot||_{L^{1}}$. The fact that $T$ takes its values in $\mathcal{S}$ is proven exactly as in the proof of Lemma~\ref{lemma:PropScheme}. We prove in Lemma~\ref{lemma:Mono} that $T$ is order preserving.
Applying the Crandall--Tartar lemma to $((u_{j}^{0}), v^{0})$ and $(\bar{u}_{j}^{0}, \bar{v})= ((u_{j}^{0}), \frac{c_{-}+c_{+}}{2})$, we obtain
 $$ \Delta x \sum_{j=-M_{c}+1}^{M_{c}} | \bar{u_{j}}^{n+1}- u_{j}^{n+1}| + m | \bar{v}^{n+1} - v^{n+1} | \leq m \left|v^{0} - \frac{u_{-}+u_{+}}{2} \right|.$$
 The result follows since $\bar{v}^{n+1}= \frac{u_{-}+u_{+}}{2}$ (see Lemma~\ref{lemma:Sym} below).
 
\end{proof}

\begin{lemma} \label{lemma:Sym}
If $g$ verifies~\eqref{eq:FluxSym} and if the initial data is
 $$\begin{cases}
 u_{j}^{0}=  u_{-} & \text{ for } j \leq 0, \\
 u_{j}^{0}=  u_{+} & \text{ for } j \geq 1,\\
 v^{0}= \frac{u_{-}+u_{+}}{2},
\end{cases}
$$
then Scheme~\eqref{eq:CoupledSchemeP} verifies $v^{n}=v^{0}$ for all integer $n$.
\end{lemma}

\begin{proof}
 We prove by induction the following stronger result: 
 $$\forall n \in \N, \forall j \leq 0,  \ v^{n}= \frac{u_{-}+u_{+}}{2} \ \ \text{ and } \ \ u_{-j}^{n}-v^{n}= v^{n}-u_{j+1}^{n} . $$
The symmetry of the initial data ensures that this is verified for $n=0$. Suppose that this is verified for some $n \geq 0$. 
Hypothesis~\eqref{eq:FluxSym} on the flux and the induction hypothesis yield
\begin{align*}
 g(u_{0}^{n}, u_{1}^{n}+ \lambda, v^{n}) &= g( v^{n}- (u_{1}^{n}+\lambda-v^{n}), v^{n}- (u_{0}^{n}-v^{n}), v^{n}) \\
 	&= g( u_{0}^{n}-\lambda , u_{1}^{n}, v^{n} ).
\end{align*}
Hence, the velocity remains constant. A similar reasoning can be applied to the fluid velocity. Let us give some details  for $j \leq -1$:
\begin{align*}
 u_{-j}^{n+1} & = u_{-j}^{n}- \mu( g(u_{-j}^{n}, u_{-(j-1)}^{n}, v^{n}) -g( u_{-(j+1)}^{n}, u_{-j}^{n}, v^{n})) \\
 	&=2 v^{n}- u_{j+1}^{n}- \mu \left[ g(v^{n}-(v^{n}-u_{-j}^{n}), (v^{n}-(v^{n}-u_{-(j-1)}^{n}), v^{n}) ) \right. \\
	& \hspace{4cm} \left. - g(v^{n}- (v^{n}-u_{-(j+1)}^{n}), v^{n}-(v^{n}-u_{-j}^{n}), v^{n}) \right] \\
	&=2 v^{n}-  \left[ u_{j+1}^{n}+ \mu [g( 2v^{n}-u_{-(j-1)}^{n},2v^{n}-u_{-j}^{n}, v^{n} ) \right. \\
	& \hspace{4cm} \left. - g( 2v^{n}-u_{-j}^{n}, 2v^{n} -u_{-(j+1)}^{n}, v^{n}) \right] \\
	&= 2 v^{n}-  \left( u_{j+1}^{n} - \mu \left[ g(u_{j+1}^{n}, u_{j+2}^{n}, v^{n}) - g(u_{j}^{n}, u_{j+1}^{n}, v^{n}) \right] \right) \\
	&= 2 v^{n+1}- u_{j+1}^{n+1},
\end{align*}
and for $j=0$:
\begin{align*}
 u_{0}^{n+1} & = u_{0}^{n}- \mu( g(u_{0}^{n}, u_{0}^{n}-\lambda, v^{n}) -g( u_{-1}^{n}, u_{0}^{n}, v^{n}) \\
 	&=2 v^{n}- u_{1}^{n}- \mu \left[ g(v^{n}-(v^{n}-u_{0}^{n}), (v^{n}-(v^{n}-u_{0}^{n}+\lambda), v^{n}) ) \right. \\
	& \hspace{4cm} \left. - g(v^{n}- (v^{n}-u_{-1}^{n}), v^{n}-(v^{n}-u_{0}^{n}), v^{n}) \right] \\
	&=2 v^{n}-  \left[ u_{1}^{n}+ \mu [g( 2v^{n}-u_{0}^{n}+\lambda, 2v^{n}-u_{0}^{n}, v^{n} ) \right. \\
	& \hspace{4cm} \left. - g( 2v^{n}-u_{0}^{n}, 2v^{n} -u_{-1}^{n}, v^{n}) \right] \\
	&= 2 v^{n}-  \left( u_{1}^{n} - \mu \left[ g(u_{1}^{n}, u_{2}^{n}, v^{n}) - g(u_{1}^{n}+\lambda, u_{1}^{n}, v^{n}) \right] \right) \\
	&= 2 v^{n+1}- u_{1}^{n+1}.
\end{align*}
 \end{proof}

\begin{lemma} \label{lemma:Mono} Suppose that  $\partial_{3} g$ is decreasing with respect to its first two arguments, and that~\eqref{gmono},  \eqref{eq:CFL} and~\eqref{dissipativity} hold. Then, if  two initial data are ordered, this order is conserved after one iteration of the scheme. More precisely, if $[(u_{j}^{n})_{j \in \Z}, v^{n}]$ and $[(\bar{u}_{j}^{n})_{j \in \Z}, \bar{v}^{n}]$  are two elements of $\mathcal{S}$ such that
$$ \forall j \in \Z, \ u_{j}^{n} \leq \bar{u}_{j}^{n} \ \ \text{ and } \ \ v^{n} \leq \bar{v}^{n}, $$
then, if $\partial_{3}g$ is decreasing with respect to its first two arguments and if
\begin{equation} \label{eq:newCFL}
  \frac{2 \Delta t}{m_{p}} \max| \partial_{3} g|   < 1,
\end{equation}
then
$$ \forall j \in \Z, \ u_{j}^{n+1} \leq \bar{u}_{j}^{n+1} \ \ \text{ and } \ \ v^{n+1} \leq \bar{v}^{n+1}. $$
\end{lemma}
 \begin{proof}
 The case where $v^{n}$ is equal to $\bar{v}^{n}$ is a straightforward. On the one hand the monotonicity assumption~\eqref{gmono} on $g$ and CFL condition~\eqref{eq:CFL} yield as usual
 $$ u_{j}^{n+1}= H_{\lambda}(u_{j-1}^{n}, u_{j}^{n}, u_{j+1}^{n}, v^{n}) \leq  H_{\lambda}(\bar{u}_{j-1}^{n}, \bar{u}_{j}^{n}, \bar{u}_{j+1}^{n}, v^{n}) = \bar{u}_{j}^{n+1}. $$
 On the other hand,
 $$ \bar{v}^{n+1}-v^{n+1}= \frac{\Delta t}{m_{p}} \big( (g_{\lambda}^{-}-g_{\lambda}^{+})( \bar{u}_{0}^{n}, \bar{u}_{1}^{n}, v^{n})-(g_{\lambda}^{-}-g_{\lambda}^{+})( u_{0}^{n}, u_{1}^{n}, v^{n}) $$
 is nonnegative by Hypothesis~\eqref{dissipativity}.
 
 We now focus on the case where $(u_{j}^{n})_{j \in \Z}$ is equal to $(\bar{u}_{j}^{n})_{j \in \Z}$ and $v^{n} \leq \bar{v}^{n}$. For $j \leq -1$ and $j \geq 2$, a straightforward computation gives that there exists $a_{j+1/2}^{n} \in [u_{j}^{n}, u_{j+1}^{n}]$ and $b_{j-1/2}^{n} \in [u_{j-1}^{n}, u_{j}^{n }]$
 \begin{align*}
 u_{j}^{n+1}- \bar{u}_{j}^{n+1} &= \mu \int_{0}^{1} \partial_{t} g(u_{j}^{n}, u_{j+1}^{n}, v^{n}+t(\bar{v}^{n} -v^{n})) -  \partial_{t} g(u_{j-1}^{n}, u_{j}^{n}, v^{n}+t(\bar{v}^{n} -v^{n})) dt \\
 	&=  \mu \int_{0}^{1}  (\bar{v}^{n}-v^{n}) \left[  \partial_{3} g(u_{j}^{n}, u_{j+1}^{n}, v^{n}+t(\bar{v}^{n} -v^{n})) \right. \\
	& \hspace{3.5cm} \left. -  \partial_{3} g(u_{j-1}^{n}, u_{j}^{n}, v^{n}+t(\bar{v}^{n} -v^{n}))\right]  dt \\
	&=\mu \int_{0}^{1}  (\bar{v}^{n}-v^{n}) \left[ \partial_{23} g(u_{j}^{n}, a_{j+1/2}^{n},v^{n}+t(\bar{v}^{n} -v^{n})) (u_{j+1}^{n}-u_{j}^{n}) \right. \\
	& \hspace{3.5cm} \left.  + \partial_{13} g(b_{j-1/2}^{n},u_{j}^{n}, v^{n}+t(\bar{v}^{n} -v^{n})) (u_{j}^{n}-u_{j-1}^{n}) \right]
 \end{align*}
Moreover, $u_{j-1}^{n} \leq u_{j}^{n} \leq u_{j+1}^{n}$ because we are considering elements of $\mathcal{S}$, thus if $\partial_{3} g$ is decreasing with respect to its first two variables, $u_{j}^{n+1} \leq \bar{u}_{j}^{n+1}$. The same reasoning extends to $j \in \{ 0, 1 \}$ because $u_{0}^{n}-u_{1}^{n} \leq \lambda$.
 Eventually, 
 \begin{align*}
 \bar{v}^{n+1}-v^{n+1}&= \bar{v}^{n}-v^{n} + \frac{\Delta t}{m_{p}}(g(u_{0}^{n}, u_{1}^{n}+ \lambda, \bar{v}^{n})-g(u_{0}^{n}, u_{1}^{n}+ \lambda, v^{n})) \\
 & \hspace{1.5cm}- \frac{\Delta t}{m_{p}}(g(u_{0}^{n}-\lambda, u_{1}^{n}, \bar{v}^{n})-g(u_{0}^{n}-\lambda, u_{1}^{n}, v^{n}) ) \\
 & \geq \left(1-\frac{2 \Delta t}{m_{p}} \max|\partial_{3} g| \right) (\bar{v}^{n}- v^{n}),
\end{align*}
which is nonnegative if~\eqref{eq:newCFL} holds.
\end{proof}

\phantomsection
\addcontentsline{toc}{section}{References} 
\bibliographystyle{alpha}
\bibliography{BiblioAnaNum}
 \end{document}

%% file: preambuleenglish.tex
\usepackage[english]{babel}
\usepackage{verbatim}
\usepackage{fancyvrb}
\usepackage[utf8x]{inputenc}
\usepackage{multicol,lipsum,float}
\usepackage[T1]{fontenc}
\usepackage{graphicx}
\usepackage{multicol}
\usepackage{graphics}
\usepackage{xcolor}
\usepackage{pgf}
\usepackage{amssymb}
\usepackage{amsmath}
\usepackage{amsfonts}
\usepackage{amssymb}
\usepackage{graphicx}
\usepackage{bm}
\usepackage{amsthm}
\usepackage{enumerate}
\usepackage{textcomp}
\usepackage{mathrsfs}
\usepackage{caption}
\usepackage{stmaryrd}
\usepackage{url}
\usepackage{psfrag}
\usepackage{hyperref}
\usepackage[framemethod=default]{mdframed} 

\newtheorem{thm}{Theorem}[section]
\newtheorem{lemma}[thm]{Lemma}
\newtheorem{prop}[thm]{Proposition}

\theoremstyle{remark}
\newtheorem{remex}[thm]{Remark}
\newenvironment{rem}
  {\pushQED{\qed}\remex}
  {\popQED\endremex}

\theoremstyle{definition}
\newtheorem{defi}[thm]{Definition}

\newcommand{\N}{\mathbb{N}}
\newcommand{\R}{\mathbb{R}}

\newcommand{\Z}{\mathbb{Z}}
\newcommand{\essinf}{\mathrm{ess\,inf}}
\newcommand{\esssup}{\mathrm{ess\,sup}}
\newcommand{\dist}{\mathrm{\, dist}}

\newcommand{\dis}{\displaystyle}
\newcommand{\eps}{\varepsilon}

\def \sign{ \mathop{\rm sgn} \nolimits} 
\newmdtheoremenv[linecolor=black, backgroundcolor= gray!20, ]{thmC}{Theorem}[section]
\newmdtheoremenv[linecolor=black, backgroundcolor= gray!20, ]{lemmaC}[thmC]{Lemma}
\newmdtheoremenv[linecolor=black, backgroundcolor= gray!20, ]{propC}[thmC]{Proposition}
\newmdtheoremenv[linecolor=black, backgroundcolor= gray!20, ]{defiC}[thmC]{Definition}
\newmdtheoremenv[linecolor=black, backgroundcolor= gray!20, ]{corC}[thmC]{Corollary}
\newmdtheoremenv[linecolor=black, backgroundcolor= gray!20, ]{remC}[thmC]{Remark}